%% file: main.tex
\documentclass[english]{article}
\usepackage[T1]{fontenc}
\usepackage[utf8]{inputenc}
\usepackage{float}
\usepackage{units}
\usepackage{dsfont}
\usepackage{amsmath}
\usepackage{amsthm}
\usepackage{amssymb}
\usepackage[authoryear]{natbib}

\makeatletter

\floatstyle{ruled}
\newfloat{algorithm}{tbp}{loa}
\providecommand{\algorithmname}{Algorithm}
\floatname{algorithm}{\protect\algorithmname}

\theoremstyle{remark}
\newtheorem{rem}{\protect\remarkname}
\theoremstyle{plain}
\newtheorem{assumption}{\protect\assumptionname}
\theoremstyle{definition}
\newtheorem{defn}{\protect\definitionname}
\theoremstyle{plain}
\newtheorem{thm}{\protect\theoremname}
\theoremstyle{plain}
\newtheorem{cor}{\protect\corollaryname}
\theoremstyle{plain}
\newtheorem{lem}{\protect\lemmaname}
\theoremstyle{plain}
\newtheorem{fact}{\protect\factname}

\@ifundefined{date}{}{\date{}}
\usepackage{blindtext}
\usepackage[preprint]{jmlr2e}
\usepackage{algorithmic}

\usepackage{lastpage}
\jmlrheading{27}{2026}{1-\pageref{LastPage}}{6/26; Revised}{}{26-0000}{Zijian Liu}

\ShortHeadings{Adam Converges in Nonsmooth Nonconvex Optimization}{Liu}
\firstpageno{1}

\makeatother

\usepackage{babel}
\providecommand{\assumptionname}{Assumption}
\providecommand{\corollaryname}{Corollary}
\providecommand{\definitionname}{Definition}
\providecommand{\factname}{Fact}
\providecommand{\lemmaname}{Lemma}
\providecommand{\remarkname}{Remark}
\providecommand{\theoremname}{Theorem}

\begin{document}
\global\long\def\A{\textsc{A}}%
\global\long\def\E{\mathbb{E}}%
\global\long\def\N{\mathbb{N}}%
\global\long\def\P{\mathbb{P}}%
\global\long\def\R{\mathbb{R}}%
\global\long\def\O{\mathcal{O}}%
\global\long\def\Z{\mathrm{Z}}%
\global\long\def\argmin{\mathrm{argmin}}%
\global\long\def\dis{\mathbb{D}}%
\global\long\def\bG{\boldsymbol{G}}%
\global\long\def\bg{\boldsymbol{g}}%
\global\long\def\bmm{\boldsymbol{m}}%
\global\long\def\bu{\boldsymbol{u}}%
\global\long\def\bv{\boldsymbol{v}}%
\global\long\def\bx{\boldsymbol{x}}%
\global\long\def\by{\boldsymbol{y}}%
\global\long\def\bz{\boldsymbol{z}}%
\global\long\def\bdelta{\boldsymbol{\delta}}%
\global\long\def\bet{\boldsymbol{\eta}}%
\global\long\def\bxi{\boldsymbol{\xi}}%
\global\long\def\bsig{\boldsymbol{\sigma}}%
\global\long\def\bzero{\boldsymbol{0}}%
\global\long\def\p{\mathrm{p}}%
\global\long\def\cp{\mathrm{q}}%
\global\long\def\defeq{\triangleq}%
\global\long\def\mydots{\dots}%
\global\long\def\d{\mathrm{d}}%
\global\long\def\1{\mathds{1}}%
\global\long\def\sgn{\mathrm{sgn}}%
\global\long\def\ex{\mathrm{Exp}}%
\global\long\def\reg{\mathrm{R}}%
\global\long\def\random{\mathrm{r}}%
\global\long\def\FTRL{\textsc{FTRL}}%
\global\long\def\EOTNC{\textsc{EO2NC}}%
\global\long\def\OTNC{\textsc{O2NC}}%
\global\long\def\DTNC{\textsc{D2NC}}%
\global\long\def\Adam{\textsc{Adam}}%
\global\long\def\AdamW{\textsc{AdamW}}%
\global\long\def\AdaGrad{\textsc{AdaGrad}}%
\global\long\def\RMSProp{\textsc{RMSProp}}%
\global\long\def\ClippedAdam{\textsc{Clipped-Adam}}%

\title{Adam Converges in Nonsmooth Nonconvex Optimization}

\author{\name Zijian Liu \email zl3067@stern.nyu.edu\\
       \addr Department of Technology, Operations, and Statistics\\
	   Leonard N. Stern School of Business\\
       New York University\\
       New York, NY 10012, USA
}

\editor{}

\maketitle

\begin{abstract}%
\textsc{Adam}  is one of the most widely implemented and influential modern optimizers. Why is it effective across different optimization problems in practice? This question arguably lies at the center of the optimization community over the last decade and has motivated a substantial body of work aimed at understanding its convergence behavior. However, existing studies have mainly focused on the convergence rate of \textsc{Adam} in smooth nonconvex optimization, which unfortunately does not adequately capture practical settings, since many real-world problems are nonsmooth, such as those arising in training neural networks. Thus, these studies cannot fully explain the popularity and empirical success of \textsc{Adam}. 

Recently, an insightful and powerful framework called Online-to-Nonconvex Conversion has opened a new way to analyze \textsc{Adam} for nonsmooth nonconvex optimization. Unfortunately, prior works along this line share two common limitations. First, all of them ignore the important bias-correction term in the original \textsc{Adam} algorithm. Second and more importantly, many of them require extra operations that are not used in \textsc{Adam}, such as a clipping step. Therefore, the convergence guarantee for the original \textsc{Adam} method still remains unclear.

In this work, we present the first finite-time analysis for the classical form of \textsc{Adam}, i.e., with the bias-correction step and without further algorithmic modifications, and prove that a randomly scaled learning rate ensures a convergence rate of $1/T^{\frac{2}{13}}$ for nonsmooth nonconvex optimization. Moreover, our result provably applies to the modern heavy-tailed noise regime, which is closer to practice. Interestingly, our theory is established under the parameter choice $\beta_1=\beta_2$, aligning with the recent empirical studies.
\end{abstract}


\input{introduction.tex}

\input{preliminary.tex}
\input{algorithms.tex}

\input{conclusion.tex}

\clearpage

\appendix
\input{appendix.tex}

\clearpage

\bibliography{ref}

\end{document}

%% file: introduction.tex
\section{Introduction}

$\Adam$, a modern optimizer that builds upon the seminal adaptive
method $\AdaGrad$ \citep{DBLP:conf/colt/McMahanS10,JMLR:v12:duchi11a}
and the famous $\RMSProp$ algorithm \citep{tieleman2012lecture},
was introduced by \citet{kingma2014adam} in 2014. Since then, it
has undoubtedly become one of the most widely implemented and influential
optimization methods. Despite its strong empirical performance, one
question has remained central to the optimization community over the
last decade:{\begin{center} \textit{Why is $\Adam$ effective across
different optimization problems in practice?} \end{center}}This question
has motivated a substantial body of work aimed at understanding the
convergence behavior, theoretical properties, and provable advantages
of $\Adam$ (see, e.g., \citet{j.2018on,NEURIPS2018_90365351,de2018convergence,chen2018on,Zou_2019_CVPR,doi:10.1137/19M1263443,guo2021novel,defossez2022a,NEURIPS2022_b6260ae5,NEURIPS2023_a3cc5012,NEURIPS2023_7ac19fdc,10.1145/3637528.3671718,zhang2025convergence}).

However, existing studies have primarily focused on the convergence
of $\Adam$ for smooth nonconvex optimization, which does not adequately
capture practical settings, since many real-world objectives, such
as those encountered in neural network training, are nonsmooth. Consequently,
these studies cannot fully explain the popularity and empirical success
of $\Adam$. Thus, developing convergence guarantees for Adam beyond
smooth functions is essential for building a theory that more faithfully
reflects its practical behavior. This gap indicates that the existing
theoretical understanding remains incomplete.

Recently, an insightful and powerful framework, known as Online-to-Nonconvex
Conversion ($\OTNC$) \citep{pmlr-v202-cutkosky23a}, has opened a
new avenue for analyzing $\Adam$ in nonsmooth nonconvex optimization.
Unfortunately, prior works along this line \citep{pmlr-v235-ahn24b,NEURIPS2024_ac8ec9b4,pmlr-v313-nguyen26b,xie2026dynamic}
share two common limitations. First, all of them ignore the important
bias-correction term in the original $\Adam$ algorithm, which is
a key component in the default implementation of $\Adam$. Second
and more importantly, many of these works require additional operations
that are not part of $\Adam$, such as a clipping step. As a result,
the algorithms analyzed in these works differ from the original Adam
method. Therefore, the convergence guarantee for the original $\Adam$
method has yet to be established.

\subsection{Our Contributions}

In this work, we analyze the classical form of $\Adam$, i.e., with
bias correction and without any further algorithmic modifications,
and establish the following results:
\begin{itemize}
\item We prove that a randomly scaled learning rate ensures that $\Adam$
converges at a rate of $1/T^{\frac{2}{13}}$ in nonsmooth nonconvex
optimization. To the best of our knowledge, this gives the first finite-time
rate for the original $\Adam$ algorithm in this setting.
\item Our analysis provably applies to the modern heavy-tailed noise regime,
which more closely reflects practice \citep{pmlr-v97-simsekli19a,NEURIPS2020_b05b57f6}.
Specifically, when the tail index $\p$ of the gradient noise satisfies
$4/3<\p\leq2$, we show that $\Adam$ converges at a rate of $1/T^{\frac{2(3\p-4)}{19\p-12}}$.
\item Interestingly, our theory is established under the parameter choice
$\beta_{1}=\beta_{2}$, which aligns with the recent empirical studies
\citep{NEURIPS2025_5bd9aa20}.
\item Finally, we extend our results to the case where the problem-dependent
parameters (e.g., the tail index $\p$) are unknown.
\end{itemize}

\subsection{Related Work}

\paragraph{$\protect\Adam$.}

$\Adam$ was proposed by \citet{kingma2014adam} as a further generalization
of $\RMSProp$ \citep{tieleman2012lecture}, both of which were inspired
by the well-known adaptive algorithm $\AdaGrad$ \citep{DBLP:conf/colt/McMahanS10,JMLR:v12:duchi11a}
originally designed for Online Convex Optimization (OCO). Since its
introduction, $\Adam$ and its variants, such as $\AdamW$ \citep{loshchilov2018decoupled},
have been observed to be highly effective across a wide range of optimization
tasks and have become dominant tools in modern machine learning, especially
for training neural networks.

Two hyperparameters that are particularly important for $\Adam$ are
$\beta_{1}\in\left[0,1\right)$ and $\beta_{2}\in\left[0,1\right)$,
the exponential decay rates for the moment estimates. \citet{kingma2014adam}
suggested setting $\beta_{1}=0.9$ and $\beta_{2}=0.999$ by default.
However, in the literature on large language models \citep{NEURIPS2020_1457c0d6,touvron2023llama,liu2024deepseek},
these two parameters are often set to $\beta_{1}=0.9$ and $\beta_{2}=0.95$.
In addition, recent works, such as \citet{NEURIPS2025_5bd9aa20},
suggest choosing $\beta_{1}=\beta_{2}$ for improved performance,
based on extensive empirical experiments.

\paragraph{Online-to-Nonconvex Conversion.}

The Online-to-Nonconvex Conversion ($\OTNC$) method, introduced by
\citet{pmlr-v202-cutkosky23a}, provides a powerful framework for
converting the regret bound of an OCO algorithm into a provable bound
for nonconvex optimization. However, one undesirable feature of $\OTNC$
is that it queries the (stochastic) gradient at a new point that differs
from the current iterate. To overcome this issue, \citet{pmlr-v235-zhang24k}
later proposed a new method called Exponentiated Online-to-Nonconvex
Conversion ($\EOTNC$). These two methods have since served as fundamental
tools for studying nonconvex optimization, for both smooth and nonsmooth
objectives, e.g., \citet{pmlr-v235-liu24bo,10.1145/3717823.3718308,pmlr-v267-ahn25a,pmlr-v313-liu26a}.

More surprisingly, it is known that $\Adam$ (with a randomly scaled
learning rate) can be derived from $\EOTNC$ when combined with the
classical Follow-the-Regularized-Leader ($\FTRL$) algorithm for OCO
\citep{pmlr-v235-ahn24b,NEURIPS2024_ac8ec9b4}. Recently, two more
works \citep{pmlr-v313-nguyen26b,xie2026dynamic} have followed this
line of research and obtained further results.

\paragraph{Nonsmooth nonconvex optimization.}

In the literature on machine learning, \citet{pmlr-v119-zhang20p}
provided the first non-asymptotic analysis for finding stationary
points of nonsmooth nonconvex functions, where the stationary point
is defined through the Goldstein $\delta$-subdifferential \citep{goldstein1977optimization}.
Since then, several subsequent works have made further contributions
\citep{NEURIPS2022_2c8d9636,pmlr-v162-tian22a,kornowski2022on,pmlr-v195-jordan23a,JMLR:v23:21-1507,tian2024no}.
Notably, the $\OTNC$ and $\EOTNC$ frameworks discussed above provide
a systematic way to construct algorithms that guarantee optimal convergence
rates \citep{pmlr-v202-cutkosky23a,pmlr-v235-zhang24k,NEURIPS2024_ac8ec9b4,pmlr-v267-ahn25a}. 

However, these works only consider deterministic gradient feedback
or a stochastic gradient oracle with finite-variance noise. Such settings
could not accurately capture practical scenarios, since the finite-variance
condition, which also covers the deterministic setting, is often considered
too optimistic; in contrast, heavy-tailed noise provides a more realistic
model, where the stochastic noise has only a finite $\p$-th moment
for $\p\in\left(1,2\right]$ \citep{pmlr-v97-simsekli19a,NEURIPS2020_b05b57f6}.
Under heavy-tailed noise, using the $\OTNC$ framework, \citet{pmlr-v235-liu24bo}
proved the first high-probability convergence bound, and \citet{pmlr-v313-liu26a}
established the first optimal in-expectation rate with a matching
lower bound.

%% file: preliminary.tex
\section{Preliminaries}

\paragraph{Notation.}

$\N$ is the set of natural numbers (excluding $0$). We let $\left[n\right]\triangleq\left\{ 1,\mydots,n\right\} ,\forall n\in\N$.
For $a,b\in\R$, we define $a\land b\defeq\min\left\{ a,b\right\} $
and $a\lor b\defeq\max\left\{ a,b\right\} $. $\R_{>0}^{d}$ (resp.
$\R_{\geq0}^{d}$) is the set of vectors in $\R^{d}$ whose coordinates
are positive (resp. non-negative). Given a vector $\boldsymbol{\Lambda}\in\R_{>0}^{d}$,
its induced inner product and norm are defined as $\left\langle \bx,\by\right\rangle _{\boldsymbol{\Lambda}}\triangleq\sum_{i=1}^{d}\bx[i]\boldsymbol{\Lambda}[i]\by[i]$
and $\left\Vert \bx\right\Vert _{\boldsymbol{\Lambda}}\triangleq\sqrt{\left\langle \bx,\bx\right\rangle _{\boldsymbol{\Lambda}}}$,
respectively. Moreover, we use the shorthand notations $\bx^{a}[i]\defeq(\bx[i])^{a}$,
$(\bx\by)[i]\defeq\bx[i]\by[i]$, and $(\bx/\by)[i]\defeq\bx[i]/\by[i]$,
for any $a\in\R$, $i\in\left[d\right]$, and $\bx,\by\in\R^{d}$,
whenever the R.H.S. is well-defined. For any differentiable function
$h$, $\nabla_{i}h$ is the partial derivative w.r.t. the $i$-th
coordinate. We use $\sgn$ to denote the sign function.

\paragraph{Objective.}

We are interested in the following optimization problem
\[
\min_{\bx\in\R^{d}}F(\bx),
\]
where $F:\R^{d}\to\R$ is differentiable and possibly nonconvex.
\begin{rem}
The differentiability of $F$ is assumed without loss of generality.
For more justification of this point, we refer the reader to \citet[Section 2]{pmlr-v202-cutkosky23a}.
\end{rem}

\paragraph{Assumptions.}

We make the following assumptions for our analysis.
\begin{assumption}[Lower boundedness]
\label{assu:lb}The objective is lower bounded, i.e., $F_{\star}\triangleq\inf_{\bx\in\R^{d}}F(\bx)\in\R$.
\end{assumption}
\begin{assumption}[Well-behavedness]
\label{assu:FTC}The objective satisfies 
\[
F(\bx)-F(\by)=\int_{0}^{1}\left\langle \nabla F(\by+t(\bx-\by)),\bx-\by\right\rangle \d t,\forall\bx,\by\in\R^{d}.
\]
\end{assumption}
\begin{assumption}[Lipschitzness]
\label{assu:lip}There exists $\bG\in\R_{>0}^{d}$ such that $\left|\nabla_{i}F(\bx)\right|\leq\bG[i],\forall\bx\in\R^{d},\forall i\in\left[d\right]$.
\end{assumption}
Assumption \ref{assu:lb} is a common condition in the nonconvex optimization
literature. Assumption \ref{assu:FTC} was introduced by \citet{pmlr-v202-cutkosky23a}.
Assumption \ref{assu:lip} is standard in nonsmooth optimization.
Indeed, one can observe that Assumption \ref{assu:FTC} is implied
by the differentiability of $F$ together with Assumption \ref{assu:lip}.
However, we still state it as a separate assumption for consistency
with prior works.
\begin{assumption}[Unbiased stochastic gradient]
\label{assu:unbias}There exists a function $\bg:\R^{d}\times\mathrm{Z}\to\R^{d}$
associated with a probability distribution $\dis$ on $\Z$ such that
$\E_{Z\sim\dis}\left[\bg(\bx;Z)\right]=\nabla F(\bx),\forall\bx\in\R^{d}$.
\end{assumption}
\begin{assumption}[Heavy-tailed noise]
\label{assu:heavy}There exist $\p\in\left(1,2\right]$ and $\bsig\in\R_{\geq0}^{d}$
such that $\E_{Z\sim\dis}\left[\left|\bxi(\bx;Z)[i]\right|^{\p}\right]\leq\bsig^{\p}[i],\forall\bx\in\R^{d},\forall i\in\left[d\right]$,
where $\bxi(\bx;Z)\triangleq\bg(\bx;Z)-\nabla F(\bx)$, with $\bg$
and $\dis$ as in Assumption \ref{assu:unbias}.
\end{assumption}
\begin{rem}
Throughout the rest of the paper, we denote the conjugate of $\p$
by $\cp\defeq\frac{\p}{\p-1}\in\left[2,+\infty\right)$.
\end{rem}
Assumption \ref{assu:unbias} is the standard unbiasedness assumption
in stochastic optimization. Assumption \ref{assu:heavy} is the heavy-tailed
noise condition, introduced in \citet{pmlr-v97-simsekli19a} and \citet{NEURIPS2020_b05b57f6},
which can better fit many modern machine learning tasks, such as training
neural networks, and also generalizes the classical finite-variance
condition (i.e., $\p=2$).

\paragraph{Metric.}

Since the objective function is neither convex nor smooth, we need
a suitable metric to measure convergence. In this work, we adopt the
following convergence criterion proposed by \citet{NEURIPS2024_ac8ec9b4}.
For more discussion on this type of convergence metric and its relation
to the classical notion of the Goldstein stationary point \citep{goldstein1977optimization},
we refer the reader to \citet{NEURIPS2024_ac8ec9b4}.
\begin{defn}[{$(\lambda,\varepsilon)$-$L_{1}$-stationary point \citep[Definition 15]{NEURIPS2024_ac8ec9b4}}]
\label{def:stationary}Suppose $F:\R^{d}\to\R$ is differentiable.
We say $\bx$ is a $(\lambda,\varepsilon)$-$L_{1}$-stationary point
of $F$ if $\left\Vert \nabla F(\bx)\right\Vert _{1}^{\left[\lambda\right]}\leq\varepsilon$,
where
\[
\left\Vert \nabla F(\bx)\right\Vert _{1}^{\left[\lambda\right]}\defeq\inf_{\substack{\P\in{\cal P}(\R^{d})\\
\E_{\by\sim\P}\left[\by\right]=\bx
}
}\left\{ \left\Vert \E_{\by\sim\P}\left[\nabla F(\by)\right]\right\Vert _{1}+\lambda\E_{\by\sim\P}\left[\left\Vert \by-\bx\right\Vert _{2}^{2}\right]\right\} ,
\]
in which ${\cal P}(\R^{d})$ denotes the set of all probability distributions
on $\R^{d}$.
\end{defn}
\begin{rem}
Recently, \citet[Definition 2]{yu2026stosignsgd} proposed the $(\lambda,\varepsilon)$-$L_{a,b}$-stationary
point of $F$, for any $1\leq a,b\leq+\infty$, defined by
\[
\left\Vert \nabla F(\bx)\right\Vert _{a,b}^{\left[\lambda\right]}\defeq\inf_{\substack{\P\in{\cal P}(\R^{d})\\
\E_{\by\sim\P}\left[\by\right]=\bx
}
}\left\{ \left\Vert \E_{\by\sim\P}\left[\nabla F(\by)\right]\right\Vert _{a}+\lambda\E_{\by\sim\P}\left[\left\Vert \by-\bx\right\Vert _{b}^{2}\right]\right\} \leq\varepsilon.
\]
This definition recovers Definition \ref{def:stationary} and the
$(\lambda,\varepsilon)$-stationary point \citep[Definition 2.2]{pmlr-v235-zhang24k}
when $(a,b)=(1,2)$ and $(a,b)=(2,2)$, respectively. In fact, by
Lemmas \ref{lem:EO2NC-stability} and \ref{lem:iterate-bound}, all
of our results in Section \ref{sec:Adam} can be generalized to $a=1$
and any $b\in\left[1,+\infty\right]$, with only the dependence on
the dimension $d$ changing. In particular, when $b=+\infty$, we
can remove the factor $d$ from our bounds. However, for consistency
with the literature, we continue to use the $(\lambda,\varepsilon)$-$L_{1}$-stationary
point, i.e., the case $(a,b)=(1,2)$.
\end{rem}

\paragraph{Two notions.}

To conclude this section, we recall two useful notions from \citet{pmlr-v267-ahn25a}
that will be used later in this work.
\begin{defn}[{Exponential moving average (EMA) of iterates \citep[Definition 3]{pmlr-v267-ahn25a}}]
\label{def:EMA}Given a discount factor $\beta_{1}\in\left[0,1\right)$
and a sequence of iterates $\left\{ \bx_{t}\right\} _{t=1}^{T}$,
the $\beta_{1}$-EMA of $\left\{ \bx_{t}\right\} _{t=1}^{T}$, denoted
by $\left\{ \bar{\bx}_{t}\right\} _{t=1}^{T}$, is defined as
\[
\bar{\bx}_{t}\defeq\frac{1-\beta_{1}}{1-\beta_{1}^{t}}\sum_{s=1}^{t}\beta_{1}^{t-s}\bx_{s},\forall t\in\left[T\right].
\]
\end{defn}
\begin{rem}
In the remainder of the paper, when $\beta_{1}$ and $\left\{ \bx_{t}\right\} _{t=1}^{T}$
are clear from context, the notation $\left\{ \bar{\bx}_{t}\right\} _{t=1}^{T}$
always denotes the $\beta_{1}$-EMA of $\left\{ \bx_{t}\right\} _{t=1}^{T}$,
as specified in Definition \ref{def:EMA}.
\end{rem}
\begin{defn}[{Random index distribution \citep[Definition 4]{pmlr-v267-ahn25a}}]
\label{def:tau}Given a discount factor $\beta_{1}\in\left[0,1\right)$,
let $\tau$ be a random index taking values in $\left[T\right]$ with
the following probability mass function:
\[
\Pr\left[\tau=t\right]\defeq\begin{cases}
\frac{1-\beta_{1}^{t}}{T} & t\in\left[T-1\right]\\
\frac{1}{1-\beta_{1}}\cdot\frac{1-\beta_{1}^{t}}{T} & t=T
\end{cases}.
\]
\end{defn}
\begin{rem}
In the remainder of the paper, when $\beta_{1}$ is clear from context,
the notation $\tau$ always denotes a random index distributed according
to the random index distribution in Definition \ref{def:tau} with
parameter $\beta_{1}$.
\end{rem}

%% file: algorithms.tex
\section{$\protect\Adam$ and Its Convergence Rates\label{sec:Adam}}

\begin{algorithm}[H]
\caption{\label{alg:Adam}$\protect\Adam$ \citep{kingma2014adam}}

\begin{algorithmic}[1]

\STATE \textbf{Input:} initial point $\bx_{1}\in\R^{d}$, two hyperparameters
$\beta_{1}\in\left[0,1\right)$ and $\beta_{2}\in\left[0,1\right)$,
learning rate $\left\{ \gamma_{t}\right\} _{t=1}^{T}$, numerical
constant $\epsilon\geq0$

\STATE Initialize $\bmm_{0}=\bzero$ and $\bv_{0}=\bzero$

\FOR{$t=1$ \textbf{to} $T$}

\STATE $\bg_{t}=\bg(\bx_{t};Z_{t})$ where $Z_{t}\sim\dis$ i.i.d.

\STATE $\bmm_{t}=\beta_{1}\bmm_{t-1}+(1-\beta_{1})\bg_{t}$

\STATE $\hat{\bmm}_{t}=\bmm_{t}/(1-\beta_{1}^{t})$

\STATE $\bv_{t}=\beta_{2}\bv_{t-1}+(1-\beta_{2})\bg_{t}^{2}$

\STATE $\hat{\bv}_{t}=\bv_{t}/(1-\beta_{2}^{t})$

\STATE $\bx_{t+1}=\bx_{t}-\gamma_{t}\hat{\bmm}_{t}/(\epsilon+\sqrt{\hat{\bv}_{t}})$

\ENDFOR

\end{algorithmic}
\end{algorithm}

We present the $\Adam$ method in Algorithm \ref{alg:Adam}. Introduced
by \citet{kingma2014adam}, $\Adam$ builds upon the seminal $\AdaGrad$
method \citep{DBLP:conf/colt/McMahanS10,JMLR:v12:duchi11a} and generalizes
the well-known $\RMSProp$ algorithm \citep{tieleman2012lecture}.
Despite its empirical success, a convergence rate for $\Adam$ in
nonsmooth nonconvex optimization is still unknown, even under finite-variance
noise, i.e., $\p=2$ in Assumption \ref{assu:heavy}. Later in this
section, we provide the first finite-time convergence rate for $\Adam$,
which further extends to heavy-tailed noise when the tail index $\p$
satisfies $\p\in\left(\nicefrac{4}{3},2\right]$, beyond the finite-variance
case.

Before moving on, we define an important quantity $\rho$, given by
the ratio of $\beta_{1}$ to $\sqrt{\beta_{2}}$, i.e.,
\begin{equation}
\rho\defeq\frac{\beta_{1}}{\sqrt{\beta_{2}}}.\label{eq:main-rho-def}
\end{equation}
$\rho$ is known to play an important role in understanding the convergence
behavior of $\Adam$ \citep{j.2018on,Zou_2019_CVPR,NEURIPS2022_b6260ae5,NEURIPS2023_7ac19fdc,10.1145/3637528.3671718}.
In particular, we are interested in the regime
\begin{equation}
\rho\in\left[0,1\right),\label{eq:main-rho-range}
\end{equation}
which aligns with the existing literature.

Now we are ready to present the main result, Theorem \ref{thm:Adam-rate},
which provides a finite-time convergence bound for $\Adam$ in nonsmooth
nonconvex optimization under heavy-tailed noise.
\begin{thm}
\label{thm:Adam-rate}Under Assumptions \ref{assu:lb}, \ref{assu:FTC},
\ref{assu:lip}, \ref{assu:unbias}, and \ref{assu:heavy}, let $\Delta\defeq F(\bx_{1})-F_{\star}$,
with the learning rate $\left\{ \gamma_{t}=\gamma\alpha_{t}\right\} _{t=1}^{T}$
where $\gamma>0$ is a constant and $\left\{ \alpha_{t}\right\} _{t=1}^{T}$
is a sequence of independent random variables satisfying $\alpha_{t}\sim\ex(1),\forall t\in\left[T\right]$,
$\Adam$ (Algorithm \ref{alg:Adam}) guarantees that
\begin{align*}
\E_{\tau,\random}\left[\left\Vert \nabla F(\bar{\bx}_{\tau})\right\Vert _{1}^{\left[\lambda\right]}\right]\lesssim & \sqrt{\frac{d\epsilon\Delta}{\gamma T}}+\sqrt{\frac{(1-\beta_{1})(1-\beta_{1}+\frac{\beta_{1}}{T})}{\rho\sqrt{1-\rho^{2}}}}\left(\frac{\left\Vert \bG\right\Vert _{1}}{\sqrt{1-\beta_{2}}}+\frac{\left\Vert \bsig\right\Vert _{1}}{(1-\beta_{2})^{\frac{1}{\p}}}\right)\\
 & +\sqrt{\left(\frac{\sqrt{1-\beta_{2}}\Delta}{\gamma T}+\frac{(1-\beta_{1})(1-\beta_{1}+\frac{\beta_{1}}{T})d\epsilon}{\rho\sqrt{(1-\rho^{2})(1-\beta_{2})}}\right)\left(\frac{\left\Vert \bG\right\Vert _{1}}{\sqrt{1-\beta_{2}}}+\frac{\left\Vert \bsig\right\Vert _{1}}{(1-\beta_{2})^{\frac{1}{\p}}}\right)}\\
 & +\frac{\left\Vert \bsig\right\Vert _{1}}{(1-\beta_{1})^{\frac{1}{\p}}T}+\frac{\beta_{1}\lambda d\gamma^{2}}{(1-\beta_{2})(1-\rho^{2})},
\end{align*}
where $\random$ denotes the randomness arising from $\left\{ \alpha_{t}\right\} _{t=1}^{T}$
and $\left\{ Z_{t}\right\} _{t=1}^{T}$.
\end{thm}
\begin{rem}
In what follows, as in Theorem \ref{thm:Adam-rate}, the notation
$\random$ always denotes the randomness arising from $\left\{ \alpha_{t}\right\} _{t=1}^{T}$
and $\left\{ Z_{t}\right\} _{t=1}^{T}$.
\end{rem}
\begin{proof}
We defer the proof to Section \ref{sec:analysis}.
\end{proof}

To the best of our knowledge, Theorem \ref{thm:Adam-rate} provides
the first convergence bound for the original $\Adam$ algorithm under
the corresponding setting. Even in the case $\p=2$, the result is
new, since none of the existing works \citep{pmlr-v235-ahn24b,NEURIPS2024_ac8ec9b4,pmlr-v313-nguyen26b,xie2026dynamic}
analyzes the original form of $\Adam$.

Before discussing Theorem \ref{thm:Adam-rate} in more detail, we
would like to mention one limitation of our result: the randomly scaled
learning rate. Although our results are proved for the default form
of $\Adam$, the learning rate in Theorem \ref{thm:Adam-rate} is
not standard, due to the multiplicative random coefficient $\alpha_{t}$,
and hence may be less compatible with practical implementations. However,
this random scaling is crucial and appears unavoidable under our current
proof technique (see Section \ref{sec:EO2NC+FTRL}). It is currently
unclear to us how to prove a finite-time convergence rate for $\Adam$
under a constant learning rate or another practical nonrandom schedule,
which we leave as an important future direction.

\subsection{Convergence Rate of $\protect\Adam$ with Known Problem-Dependent
Parameters}

Now, to better understand the rate in Theorem \ref{thm:Adam-rate},
we provide the following corollary, assuming that the value of the
tail index $\p$ and all other problem-dependent parameters are known.
\begin{cor}
\label{cor:Adam-known-rate-optimal-lr}Under the same setting as in
Theorem \ref{thm:Adam-rate}, with
\[
\gamma=\left\{ \frac{(1-\beta_{2})^{2}(1-\rho^{2})^{2}\Delta}{\beta_{1}^{2}\lambda^{2}dT}\left[\epsilon\lor\left(\frac{\left\Vert \bG\right\Vert _{1}}{d}+\frac{\left\Vert \bsig\right\Vert _{1}}{d(1-\beta_{2})^{\frac{2-\p}{2\p}}}\right)\right]\right\} ^{\frac{1}{5}},
\]
$\Adam$ (Algorithm \ref{alg:Adam}) guarantees that
\begin{align*}
\E_{\tau,\random}\left[\left\Vert \nabla F(\bar{\bx}_{\tau})\right\Vert _{1}^{\left[\lambda\right]}\right]\lesssim & \left[\frac{\beta_{1}\lambda d\Delta^{2}}{(1-\beta_{2})(1-\rho^{2})T^{2}}\left(d^{2}\epsilon^{2}+\left\Vert \bG\right\Vert _{1}^{2}+\frac{\left\Vert \bsig\right\Vert _{1}^{2}}{(1-\beta_{2})^{\frac{2-\p}{\p}}}\right)\right]^{\frac{1}{5}}\\
 & +\sqrt{\frac{(1-\beta_{1})(1-\beta_{1}+\frac{\beta_{1}}{T})}{\rho\sqrt{1-\rho^{2}}}}\left(\frac{\left\Vert \bG\right\Vert _{1}}{\sqrt{1-\beta_{2}}}+\frac{\left\Vert \bsig\right\Vert _{1}}{(1-\beta_{2})^{\frac{1}{\p}}}\right)\\
 & +\sqrt{\frac{(1-\beta_{1})(1-\beta_{1}+\frac{\beta_{1}}{T})d\epsilon}{\rho\sqrt{(1-\rho^{2})(1-\beta_{2})}}\left(\frac{\left\Vert \bG\right\Vert _{1}}{\sqrt{1-\beta_{2}}}+\frac{\left\Vert \bsig\right\Vert _{1}}{(1-\beta_{2})^{\frac{1}{\p}}}\right)}+\frac{\left\Vert \bsig\right\Vert _{1}}{(1-\beta_{1})^{\frac{1}{\p}}T}.
\end{align*}
Furthermore, with any $\epsilon\leq\frac{\left\Vert \bG\right\Vert _{1}}{d}+\frac{\left\Vert \bsig\right\Vert _{1}}{d(1-\beta_{2})^{\frac{2-\p}{2\p}}}$,
the above bound reduces to
\begin{align*}
\E_{\tau,\random}\left[\left\Vert \nabla F(\bar{\bx}_{\tau})\right\Vert _{1}^{\left[\lambda\right]}\right]\lesssim & \left[\frac{\beta_{1}\lambda d\Delta^{2}}{(1-\beta_{2})(1-\rho^{2})T^{2}}\left(\left\Vert \bG\right\Vert _{1}^{2}+\frac{\left\Vert \bsig\right\Vert _{1}^{2}}{(1-\beta_{2})^{\frac{2-\p}{\p}}}\right)\right]^{\frac{1}{5}}\\
 & +\sqrt{\frac{(1-\beta_{1})(1-\beta_{1}+\frac{\beta_{1}}{T})}{\rho\sqrt{1-\rho^{2}}}}\left(\frac{\left\Vert \bG\right\Vert _{1}}{\sqrt{1-\beta_{2}}}+\frac{\left\Vert \bsig\right\Vert _{1}}{(1-\beta_{2})^{\frac{1}{\p}}}\right)+\frac{\left\Vert \bsig\right\Vert _{1}}{(1-\beta_{1})^{\frac{1}{\p}}T}.
\end{align*}
\end{cor}
\begin{proof}
With Theorem \ref{thm:Adam-rate}, a direct calculation yields Corollary
\ref{cor:Adam-known-rate-optimal-lr}.
\end{proof}

Corollary \ref{cor:Adam-known-rate-optimal-lr} establishes the convergence
rate of $\Adam$ with the optimally tuned learning rate $\gamma$
and a sufficiently small $\epsilon$ for fixed $\beta_{1}$ and $\beta_{2}$,
which is broadly consistent with common practice, as people typically
tune $\gamma$ carefully and choose $\epsilon$ to be small, e.g.,
$10^{-8}$, while keeping $\beta_{1}$ and $\beta_{2}$ fixed.

Next, we show that, when $\beta_{1}=\beta_{2}=\beta\in\left[0,1\right)$
and $\p\in\left(\nicefrac{4}{3},2\right]$, $\Adam$ guarantees both
a finite-time convergence rate (with a properly chosen $\beta$) and
asymptotic convergence (as $\beta\to1$).
\begin{cor}
\label{cor:Adam-known-rate-optimal-beta}Under the same setting as
in Corollary \ref{cor:Adam-known-rate-optimal-lr}, i.e., optimally
tuned $\gamma$ and sufficiently small $\epsilon$, with
\[
\beta_{1}=\beta_{2}=\beta\in\left[0,1\right),
\]
$\Adam$ (Algorithm \ref{alg:Adam}) guarantees that
\begin{align*}
\E_{\tau,\random}\left[\left\Vert \nabla F(\bar{\bx}_{\tau})\right\Vert _{1}^{\left[\lambda\right]}\right]\lesssim & \left[\frac{\beta\lambda d\Delta^{2}}{(1-\beta)^{2}T^{2}}\left(\left\Vert \bG\right\Vert _{1}^{2}+\frac{\left\Vert \bsig\right\Vert _{1}^{2}}{(1-\beta)^{\frac{2-\p}{\p}}}\right)\right]^{\frac{1}{5}}\\
 & +\sqrt{\frac{\sqrt{1-\beta}(1-\beta+\frac{\beta}{T})}{\sqrt{\beta}}}\left(\frac{\left\Vert \bG\right\Vert _{1}}{\sqrt{1-\beta}}+\frac{\left\Vert \bsig\right\Vert _{1}}{(1-\beta)^{\frac{1}{\p}}}\right)+\frac{\left\Vert \bsig\right\Vert _{1}}{(1-\beta)^{\frac{1}{\p}}T}.
\end{align*}
Furthermore, suppose that $\p\in\left(\nicefrac{4}{3},2\right]$:
\begin{itemize}
\item with $\beta=1-\min\left\{ \frac{1}{2},\max\left\{ \left(\frac{\sqrt{\lambda d}\Delta/\left\Vert \bG+\bsig\right\Vert _{1}^{\frac{3}{2}}}{T}\right)^{\frac{8\p}{19\p-12}},\frac{1}{T+1}\right\} \right\} $,
the above bound reduces to
\[
\E_{\tau,\random}\left[\left\Vert \nabla F(\bar{\bx}_{\tau})\right\Vert _{1}^{\left[\lambda\right]}\right]\lesssim\left(\frac{\sqrt{\lambda d}\Delta\left\Vert \bG+\bsig\right\Vert _{1}}{T}\right)^{\frac{2}{5}}+\frac{\left\Vert \bG+\bsig\right\Vert _{1}}{T^{\frac{3\p-4}{4\p}}}+\left(\frac{\sqrt{\lambda d}\Delta\left\Vert \bG+\bsig\right\Vert _{1}^{\frac{5\p}{3\p-4}}}{T}\right)^{\frac{2(3\p-4)}{19\p-12}},
\]
which implies a rate of $1/T^{\frac{2}{13}}$ under the finite-variance
noise, i.e., $\p=2$.
\item with any $\beta\in\left[0,1\right)$, the above bound guarantees that
\[
\limsup_{T\to\infty}\E_{\tau,\random}\left[\left\Vert \nabla F(\bar{\bx}_{\tau})\right\Vert _{1}^{\left[\lambda\right]}\right]\lesssim\frac{(1-\beta)^{\frac{1}{4}}\left\Vert \bG\right\Vert _{1}+(1-\beta)^{\frac{3}{4}-\frac{1}{\p}}\left\Vert \bsig\right\Vert _{1}}{\beta^{\frac{1}{4}}},
\]
which vanishes as $\beta\to1$.
\end{itemize}
\end{cor}
\begin{rem}
We note that the requirement of sufficiently small $\epsilon$ is
actually unnecessary for Corollary \ref{cor:Adam-known-rate-optimal-beta}
and is kept only to simplify the expression. For any $\epsilon\geq0$,
one can prove similar results for both the finite-time rate and asymptotic
convergence, which we leave to the interested reader.
\end{rem}
\begin{rem}
For finite-time convergence, a simpler choice $\beta=1-1/(T+1)^{\frac{8\p}{19\p-12}}$
also leads to a provable $1/T^{\frac{2(3\p-4)}{19\p-12}}$ rate, but
with different dependence on the other parameters.
\end{rem}
\begin{proof}
Corollary \ref{cor:Adam-known-rate-optimal-beta} follows directly
from Corollary \ref{cor:Adam-known-rate-optimal-lr}.
\end{proof}

There are several points about Corollary \ref{cor:Adam-known-rate-optimal-beta}
that we would like to discuss. First, the choice $\beta_{1}=\beta_{2}=\beta$
is inspired by the recent comprehensive empirical study of \citet{NEURIPS2025_5bd9aa20},
which suggests that $\beta_{1}=\beta_{2}$ leads to observed improvements
in the performance of $\Adam$. However, we emphasize that this choice
is not the only one that ensures convergence of $\Adam$. Nevertheless,
it yields the best finite-time bound that we have obtained so far.

Next, for the non-asymptotic rate, Corollary \ref{cor:Adam-known-rate-optimal-beta}
implies that the sample complexity of $\Adam$ to find a $(\lambda,\varepsilon)$-$L_{1}$-stationary
point is at most (assuming $\varepsilon\lesssim\left\Vert \bG+\bsig\right\Vert _{1}$
to simplify the discussion) $\frac{\left\Vert \bG+\bsig\right\Vert _{1}^{\frac{4\p}{3\p-4}}}{\varepsilon^{\frac{4\p}{3\p-4}}}+\frac{\sqrt{\lambda d}\Delta\left\Vert \bG+\bsig\right\Vert _{1}^{\frac{5\p}{3\p-4}}}{\varepsilon^{\frac{19\p-12}{2(3\p-4)}}}$,
which is the first theoretical guarantee for the original $\Adam$
algorithm in nonsmooth nonconvex optimization. In the widely considered
finite-variance case, i.e., $\p=2$, the above sample complexity reduces
to $\frac{\sqrt{\lambda d}\Delta\left\Vert \bG+\bsig\right\Vert _{1}^{5}}{\varepsilon^{\frac{13}{2}}}+\frac{\left\Vert \bG+\bsig\right\Vert _{1}^{4}}{\varepsilon^{4}}$,
which, however, is worse than the currently best-known bound $\frac{\sqrt{\lambda d}\Delta\left\Vert \bG+\bsig\right\Vert _{1}^{2}}{\varepsilon^{\frac{7}{2}}}+\frac{\left\Vert \bG+\bsig\right\Vert _{1}^{3}}{\varepsilon^{3}}$
proved for $\ClippedAdam$ \citep{NEURIPS2024_ac8ec9b4}. It remains
unclear whether the original $\Adam$ algorithm, that is, without
the extra clipping operation and with the bias-correction step, can
achieve this improved bound and thereby match the guarantee for $\ClippedAdam$.

Moreover, for asymptotic convergence, our result is also new for nonsmooth
nonconvex optimization and, importantly, is consistent with the existing
theoretical finding on $\Adam$ for smooth nonconvex optimization
\citep{NEURIPS2022_b6260ae5}\footnote{Strictly speaking, the noise model studied in \citet{NEURIPS2022_b6260ae5}
is not directly comparable with ours, due to differences in the objectives
considered in the two papers. However, for simplicity, we omit the
detailed discussion on this issue.}, which suggests $\beta_{2}\to1$ leads to provable convergence of
$\Adam$ when $T$ is sufficiently large, corresponding to taking
$\beta\to1$ in our setting with $\beta_{1}=\beta_{2}=\beta$. Notably,
as with the finite-time convergence rate, our result also extends
to the heavy-tailed noise regime where $\frac{4}{3}<\p\leq2$.

Lastly, we would like to mention that the threshold $\p=\frac{4}{3}$
has appeared in different recent works on heavy-tailed stochastic
optimization \citep{liu2026can,liu2026expectation}. It would be interesting
to investigate whether this value is an inherent barrier for such
problems.

\subsection{Convergence Rate of $\protect\Adam$ with Unknown Problem-Dependent
Parameters}

The previous corollaries are established under the assumption that
problem-dependent parameters are known a priori, which may not be
practical. Therefore, we turn our attention to the setting in which
the tail index $\p$ and other problem-dependent parameters are unknown.
Moreover, to save space, we focus directly on the case $\beta_{1}=\beta_{2}=\beta$
as before and establish the following corollary.
\begin{cor}
\label{cor:Adam-unknown-rate}Under the same setting as in Theorem
\ref{thm:Adam-rate}, with $\beta_{1}=\beta_{2}=\beta\in\left[0,1\right)$,
$\epsilon\leq\frac{\left\Vert \bG+\bsig\right\Vert _{1}}{d}$, and
$\gamma=\left[\frac{(1-\beta)^{4}C}{\beta^{2}\lambda^{2}d^{2}T}\right]^{\frac{1}{5}}$
where $C>0$, $\Adam$ (Algorithm \ref{alg:Adam}) guarantees that
\begin{align*}
\E_{\tau,\random}\left[\left\Vert \nabla F(\bar{\bx}_{\tau})\right\Vert _{1}^{\left[\lambda\right]}\right]\lesssim & \left[\frac{\beta\lambda d}{(1-\beta)^{\frac{3}{4}+\frac{5}{2\p}}T^{2}}\sqrt{C^{4}+\frac{\Delta^{5}\left\Vert \bG+\bsig\right\Vert _{1}^{5}}{C}}\right]^{\frac{1}{5}}\\
 & +\sqrt{\frac{\sqrt{1-\beta}(1-\beta+\frac{\beta}{T})}{\sqrt{\beta}}}\left(\frac{\left\Vert \bG\right\Vert _{1}}{\sqrt{1-\beta}}+\frac{\left\Vert \bsig\right\Vert _{1}}{(1-\beta)^{\frac{1}{\p}}}\right)+\frac{\left\Vert \bsig\right\Vert _{1}}{(1-\beta)^{\frac{1}{\p}}T},
\end{align*}
Furthermore, suppose that $\p\in\left(\nicefrac{4}{3},2\right]$:
\begin{itemize}
\item with $\beta=1-(T+1)^{-\frac{8}{13}}$, the above bound reduces to
\[
\E_{\tau,\random}\left[\left\Vert \nabla F(\bar{\bx}_{\tau})\right\Vert _{1}^{\left[\lambda\right]}\right]\lesssim\left[\frac{\lambda d}{T^{\frac{20(\p-1)}{13\p}}}\sqrt{C^{4}+\frac{\Delta^{5}\left\Vert \bG+\bsig\right\Vert _{1}^{5}}{C}}\right]^{\frac{1}{5}}+\frac{\left\Vert \bG\right\Vert _{1}}{T^{\frac{2}{13}}}+\frac{\left\Vert \bsig\right\Vert _{1}}{T^{\frac{2(3\p-4)}{13\p}}},
\]
which implies a rate of $1/T^{\frac{2}{13}}$ under the finite-variance
noise, i.e., $\p=2$.
\item with any $\beta\in\left[0,1\right)$, the above bound guarantees that
\[
\limsup_{T\to\infty}\E_{\tau,\random}\left[\left\Vert \nabla F(\bar{\bx}_{\tau})\right\Vert _{1}^{\left[\lambda\right]}\right]\lesssim\frac{(1-\beta)^{\frac{1}{4}}\left\Vert \bG\right\Vert _{1}+(1-\beta)^{\frac{3}{4}-\frac{1}{\p}}\left\Vert \bsig\right\Vert _{1}}{\beta^{\frac{1}{4}}},
\]
which vanishes as $\beta\to1$.
\end{itemize}
\end{cor}
\begin{rem}
To clarify, although the value of $\left\Vert \bG+\bsig\right\Vert _{1}$
is assumed unknown in this subsection, we allow $\epsilon$ to satisfy
the inequality involving $\left\Vert \bG+\bsig\right\Vert _{1}$,
since $\epsilon$ is typically very small in practice and can indeed
satisfy the condition.
\end{rem}
\begin{proof}
Corollary \ref{cor:Adam-unknown-rate} follows from Theorem \ref{thm:Adam-rate}
directly.
\end{proof}

Like Corollary \ref{cor:Adam-known-rate-optimal-beta}, Corollary
\ref{cor:Adam-unknown-rate} establishes both non-asymptotic and asymptotic
guarantees. The key difference is that the choices of $\beta$ and
$\gamma$ no longer depend on the tail index $\p$ or on the problem-dependent
quantities $\Delta$, $\bG$, and $\bsig$. The free parameter $C>0$
in the learning rate does not affect the order of $T$ in the final
rate. If the problem-dependent quantities were known, $C$ could be
tuned to balance the two terms $C^{4}$ and $\frac{\Delta^{5}\left\Vert \bG+\bsig\right\Vert _{1}^{5}}{C}$.
However, as expected, the finite-time rate weakens to $1/T^{\frac{2(3\p-4)}{13\p}}$,
since the value of $\p$ is no longer known. Finally, the asymptotic
result is the same as in Corollary \ref{cor:Adam-known-rate-optimal-beta}
and retains the residual term $\frac{(1-\beta)^{\frac{1}{4}}\left\Vert \bG\right\Vert _{1}+(1-\beta)^{\frac{3}{4}-\frac{1}{\p}}\left\Vert \bsig\right\Vert _{1}}{\beta^{\frac{1}{4}}}$.

\section{Understanding Adam via the Online-to-Nonconvex Conversion\label{sec:EO2NC+FTRL}}

In this section, we first present two existing algorithms from the
literature and then provide basic theoretical results for both of
them. Finally, we recall an important fact: when combined, these two
algorithms recover $\Adam$.

\subsection{Exponentiated Online-to-Nonconvex Conversion}

\begin{algorithm}[H]
\caption{\label{alg:EO2NC}Exponentiated Online-to-Nonconvex Conversion \citep{pmlr-v235-zhang24k,NEURIPS2024_ac8ec9b4}}

\begin{algorithmic}[1]

\STATE \textbf{Input:} OCO algorithm $\A$, initial point $\bx_{1}\in\R^{d}$,
hyperparameter $\beta_{1}\in\left[0,1\right)$

\FOR{$t=1$ \textbf{to} $T$}

\STATE $\bg_{t}=\bg(\bx_{t};Z_{t})$ where $Z_{t}\sim\dis$ i.i.d.

\STATE Send $\ell_{t}(\bdelta)\defeq\left\langle \beta_{1}^{-t}\bg_{t},\bdelta\right\rangle $
to $\A$

\STATE Receive $\bdelta_{t+1}$ from $\A$

\STATE $\bx_{t+1}=\bx_{t}+\alpha_{t}\bdelta_{t+1}$ where $\alpha_{t}\sim\ex(1)$
i.i.d.

\ENDFOR

\end{algorithmic}
\end{algorithm}

\begin{rem}
The $\EOTNC$ framework presented in Algorithm \ref{alg:EO2NC} differs
slightly from its original form in \citet{pmlr-v235-zhang24k} and
instead follows the version given in \citet{NEURIPS2024_ac8ec9b4},
where it is referred to as Discounted-to-Nonconvex Conversion ($\DTNC$).
However, these two methods are essentially interchangeable in the
following sense: $\EOTNC$ recovers $\DTNC$ by setting the regularization
term in $\EOTNC$ to zero, and $\DTNC$ recovers $\EOTNC$ by using
an OCO algorithm $\A$ that incorporates a regularization term. Therefore,
in this work, we retain the name $\EOTNC$ for simplicity.
\end{rem}
We present the Exponentiated Online-to-Nonconvex Conversion ($\EOTNC$)
method in Algorithm \ref{alg:EO2NC}. The $\EOTNC$ framework was
introduced by \citet{pmlr-v235-zhang24k} and \citet{NEURIPS2024_ac8ec9b4},
building on the original Online-to-Nonconvex Conversion ($\OTNC$)
framework of \citet{pmlr-v202-cutkosky23a}. The key distinction between
$\EOTNC$ and $\OTNC$ is that $\EOTNC$ always queries the stochastic
gradient at the current iterate $\bx_{t}$, whereas $\OTNC$ evaluates
the stochastic gradient at a point different from $\bx_{t}$. For
more details on the differences between $\EOTNC$ and $\OTNC$ and
the corresponding design principles, we refer the reader to \citet{pmlr-v235-zhang24k}
(or \citet{NEURIPS2024_ac8ec9b4}) and \citet{pmlr-v202-cutkosky23a},
respectively.

Next, we state two useful results for $\EOTNC$ in Lemmas \ref{lem:EO2NC-gradient}
and \ref{lem:EO2NC-stability}.
\begin{lem}
\label{lem:EO2NC-gradient}Under Assumptions \ref{assu:lb}, \ref{assu:FTC},
\ref{assu:unbias}, and \ref{assu:heavy}, let $\Delta\defeq F(\bx_{1})-F_{\star}$,
$U>0$ be a constant, and $\left\{ \bu_{t}\right\} _{t=1}^{T}$ be
a sequence of vectors defined as 
\begin{equation}
\bu_{t}[i]\defeq-U\cdot\sgn\left(\sum_{s=1}^{t}\beta_{1}^{-s}\nabla_{i}F(\bx_{s})\right),\forall i\in\left[d\right],\forall t\in\left[T\right],\label{eq:main-u-def}
\end{equation}
then $\EOTNC$ (Algorithm \ref{alg:EO2NC}) guarantees that
\[
\E_{\tau,\random}\left[\left\Vert \E_{\by_{\tau}}\left[\nabla F(\by_{\tau})\right]\right\Vert _{1}\right]\leq\frac{\Delta}{UT}+\sum_{t=1}^{T}\frac{1-\beta_{1}\1\left[t\neq T\right]}{UT}\E_{\random}\left[\beta_{1}^{t}\reg_{t}(\bu_{t})\right]+\frac{2\left\Vert \bsig\right\Vert _{1}}{(1-\beta_{1})^{\frac{1}{\p}}T}+2(1-\beta_{1})^{\frac{1}{\cp}}\left\Vert \bsig\right\Vert _{1},
\]
where
\begin{itemize}
\item $\left\{ \by_{t}\right\} _{t=1}^{T}$ is a sequence of random vectors
satisfying $\Pr\left[\by_{t}=\bx_{s}\right]=\frac{1-\beta_{1}}{1-\beta_{1}^{t}}\beta_{1}^{t-s},\forall s\in\left[t\right],\forall t\in\left[T\right]$.
\item $\reg_{t}(\bu)\defeq\sum_{s=1}^{t}\ell_{s}(\bdelta_{s})-\ell_{s}(\bu)=\sum_{s=1}^{t}\left\langle \beta_{1}^{-s}\bg_{s},\bdelta_{s}-\bu\right\rangle ,\forall\bu\in\R^{d},\forall t\in\left[T\right]$.
\end{itemize}
\end{lem}
\begin{proof}
See Appendix \ref{sec:EO2NC}.
\end{proof}

Lemma \ref{lem:EO2NC-gradient} provides the first theoretical guarantee
for $\EOTNC$ under heavy-tailed noise (i.e., Assumption \ref{assu:heavy}),
extending the existing result for finite-variance noise of \citet[Lemma 7]{NEURIPS2024_ac8ec9b4}.
This result can be viewed as a counterpart of the convergence bound
for the original $\OTNC$ method under heavy-tailed gradient noise,
which was recently proved by \citet{pmlr-v313-liu26a}.

Lemma \ref{lem:EO2NC-stability} below bounds the weighted sum of
the variances of $\left\{ \by_{t}\right\} _{t=1}^{T}$, which is required
to obtain convergence guarantees in terms of the $(\lambda,\varepsilon)$-$L_{1}$-stationary
point (see Definition \ref{def:stationary}).
\begin{lem}
\label{lem:EO2NC-stability}Let $\left\Vert \cdot\right\Vert $ be
any norm on $\R^{d}$, then $\EOTNC$ (Algorithm \ref{alg:EO2NC})
guarantees that
\[
\E_{\tau,\random,\by_{\tau}}\left[\left\Vert \by_{\tau}-\bar{\bx}_{\tau}\right\Vert ^{2}\right]\leq\frac{4\beta_{1}}{(1-\beta_{1})^{2}T}\sum_{t=1}^{T}\E_{\random}\left[\left\Vert \bdelta_{t+1}\right\Vert ^{2}\right],
\]
where $\left\{ \by_{t}\right\} _{t=1}^{T}$ is a sequence of random
vectors satisfying $\Pr\left[\by_{t}=\bx_{s}\right]=\frac{1-\beta_{1}}{1-\beta_{1}^{t}}\beta_{1}^{t-s},\forall s\in\left[t\right],\forall t\in\left[T\right]$.
\end{lem}
\begin{proof}
By \citet[Lemma A.5]{pmlr-v267-ahn25a}\footnote{Although the lemma is stated in terms of the $L_{2}$ norm, upon closer
examination, it actually holds for any norm.}, we have $\E_{\tau,\by_{\tau}}\left[\left\Vert \by_{\tau}-\bar{\bx}_{\tau}\right\Vert ^{2}\right]\leq\frac{2\beta_{1}}{(1-\beta_{1})^{2}T}\sum_{t=2}^{T}\left\Vert \bx_{t}-\bx_{t-1}\right\Vert ^{2}$.
We obtain the desired inequality by substituting $\bx_{t}-\bx_{t-1}=\alpha_{t-1}\bdelta_{t}$,
taking expectations with respect to $\random$ on both sides, and
using $\E_{\random}\left[\alpha_{t-1}^{2}\left\Vert \bdelta_{t}\right\Vert ^{2}\right]\overset{(a)}{=}\E_{\random}\left[\alpha_{t-1}^{2}\right]\E_{\random}\left[\left\Vert \bdelta_{t}\right\Vert ^{2}\right]\overset{(b)}{=}2\E_{\random}\left[\left\Vert \bdelta_{t}\right\Vert ^{2}\right]$,
where $(a)$ follows from the independence of $\alpha_{t-1}$ and
$\bdelta_{t}$, and $(b)$ is by the fact that the second moment of
an $\ex(1)$ random variable is $2$.
\end{proof}

\subsection{Follow-the-Regularized-Leader}

\begin{algorithm}[H]
\caption{\label{alg:FTRL}Follow-the-Regularized-Leader}

\begin{algorithmic}[1]

\FOR{$t=1$ \textbf{to} $T$}

\STATE Choose $\bet_{t}\in\R_{>0}^{d}$

\STATE $\bdelta_{t+1}=\argmin_{\bdelta\in\R^{d}}\frac{1}{2}\left\Vert \bdelta\right\Vert _{\bet_{t}^{-1}}^{2}+\sum_{s=1}^{t}\ell_{s}(\bdelta)$

\ENDFOR

\end{algorithmic}
\end{algorithm}

\begin{rem}
\label{rem:FTRL-delta-eta}By convention, we set $\bdelta_{1}\defeq\bzero$.
Note that this definition is compatible with the update rule of $\FTRL$
for any choice of $\bet_{0}\in\R_{>0}^{d}$. In particular, throughout
the paper, we take $\bet_{0}\defeq\bet_{1}$.
\end{rem}
We present the classical Follow-the-Regularized-Leader ($\FTRL$)
method in Algorithm \ref{alg:FTRL}, a powerful online learning method
with broad applications (see, e.g., \citet{DBLP:conf/colt/Gordon99,DBLP:conf/colt/Shalev-ShwartzS06,DBLP:conf/colt/AbernethyHR08,DBLP:conf/colt/HazanK08,nesterov2009primal,orabona2015generalized,JMLR:v18:14-428}).
Note that $\FTRL$ can be stated more generally by replacing the regularizer
$\frac{1}{2}\left\Vert \bdelta\right\Vert _{\bet_{t}^{-1}}^{2}$ with
an abstract function $R_{t}(\bdelta)$. Since the quadratic regularization
is important for our purpose of recovering $\Adam$, we keep the current
form. For further details on $\FTRL$, we refer the interested reader
to, for example, \citet[Section 7]{orabona2019modern}.

Lemma \ref{lem:FTRL-regret} below provides a regret bound for $\FTRL$.
The proof is standard in the literature, but we include it in Appendix
\ref{sec:FTRL} to make the paper self-contained.
\begin{lem}
\label{lem:FTRL-regret}Suppose $\left\{ \ell_{t}\right\} _{t=1}^{T}$
is differentiable and convex for any $t\in\left[T\right]$ and $\left\{ \bet_{t}[i]\right\} _{t=1}^{T}$
is nonincreasing for any $i\in\left[d\right]$, then $\FTRL$ (Algorithm
\ref{alg:FTRL}) guarantees that, for any $t\in\left[T\right]$ and
$\bu\in\R^{d}$,
\[
\reg_{t}(\bu)\leq\frac{1}{2}\left\Vert \bu\right\Vert _{\bet_{t}^{-1}}^{2}+\sum_{i=1}^{d}\sum_{s=1}^{t}\left\Vert \bdelta_{s}-\bdelta_{s+1}\right\Vert _{\infty}\left|\nabla_{i}\ell_{s}(\bdelta_{s})\right|\land\frac{\bet_{s-1}[i]\left(\nabla_{i}\ell_{s}(\bdelta_{s})\right)^{2}}{2},
\]
where $\reg_{t}(\bu)\defeq\sum_{s=1}^{t}\ell_{s}(\bdelta_{s})-\ell_{s}(\bu),\forall t\in\left[T\right]$.
\end{lem}
\begin{rem}
Differentiability is actually unnecessary for Lemma \ref{lem:FTRL-regret},
and one can replace $\nabla\ell_{s}(\bdelta_{s})$ with a subgradient
of $\ell_{s}$ at $\bdelta_{s}$. For the sake of notation, we keep
the differentiability assumption.
\end{rem}
\begin{proof}
See Appendix \ref{sec:FTRL}.
\end{proof}

\subsection{$\protect\Adam=\protect\EOTNC+\protect\FTRL$}

Finally, we recall a known result noted in the literature \citep{pmlr-v235-ahn24b,NEURIPS2024_ac8ec9b4}
that $\Adam$ with a randomly scaled learning rate $\left\{ \gamma_{t}=\gamma\alpha_{t}\right\} _{t=1}^{T}$,
where $\gamma>0$ and $\alpha_{t}\sim\ex(1)$ are i.i.d., coincides
with $\EOTNC$ when the input OCO algorithm is $\A=\FTRL$ equipped
with the stepsize sequence $\left\{ \bet_{t}\right\} _{t=1}^{T}$
defined by
\begin{equation}
\bet_{t}[i]\defeq\frac{\gamma\beta_{1}^{t}(1-\beta_{1})/(1-\beta_{1}^{t})}{\epsilon+\sqrt{(\sum_{s=1}^{t}\beta_{2}^{t-s}\bg_{s}^{2}[i])(1-\beta_{2})/(1-\beta_{2}^{t})}},\forall i\in\left[d\right],\forall t\in\left[T\right].\label{eq:main-vector-eta-def}
\end{equation}
Formally, the following fact holds.
\begin{fact}
\label{fact:equivalence}Let $\bx_{1}\in\R^{d}$, $\beta_{1}\in\left[0,1\right)$,
$\beta_{2}\in\left[0,1\right)$, $\gamma>0$, and $\epsilon\geq0$
be fixed and used as inputs whenever needed. The following two algorithms
produce the same trajectory $\left\{ \bx_{t}\right\} _{t=1}^{T+1}$:
\begin{itemize}
\item $\Adam$ with a randomly scaled learning rate $\left\{ \gamma_{t}=\gamma\alpha_{t}\right\} _{t=1}^{T}$,
where $\alpha_{t}\sim\ex(1)$ are i.i.d.
\item $\EOTNC$ with the OCO algorithm $\A=\FTRL$, in which the stepsize
$\left\{ \bet_{t}\right\} _{t=1}^{T}$ is defined in (\ref{eq:main-vector-eta-def}).
\end{itemize}
\end{fact}
\begin{proof}
Fix $t\in\left[T\right]$. By the update rule of $\FTRL$ and $\ell_{t}(\bdelta)=\left\langle \beta_{1}^{-t}\bg_{t},\bdelta\right\rangle $
in $\EOTNC$, we have
\[
\bdelta_{t+1}=\argmin_{\bdelta\in\R^{d}}\frac{1}{2}\left\Vert \bdelta\right\Vert _{\bet_{t}^{-1}}^{2}+\sum_{s=1}^{t}\ell_{s}(\bdelta)=\argmin_{\bdelta\in\R^{d}}\frac{1}{2}\left\Vert \bdelta\right\Vert _{\bet_{t}^{-1}}^{2}+\left\langle \sum_{s=1}^{t}\beta_{1}^{-s}\bg_{s},\bdelta\right\rangle ,
\]
which implies that
\[
\bdelta_{t+1}[i]=-\bet_{t}[i]\sum_{s=1}^{t}\beta_{1}^{-s}\bg_{s}[i],\forall i\in\left[d\right].
\]
In particular, when $\bet_{t}[i]=\frac{\gamma\beta_{1}^{t}(1-\beta_{1})/(1-\beta_{1}^{t})}{\epsilon+\sqrt{(\sum_{s=1}^{t}\beta_{2}^{t-s}\bg_{s}^{2}[i])(1-\beta_{2})/(1-\beta_{2}^{t})}}$
as given in (\ref{eq:main-vector-eta-def}), we obtain that
\begin{equation}
\bdelta_{t+1}[i]=-\frac{\gamma(\sum_{s=1}^{t}\beta_{1}^{t-s}\bg_{s}[i])(1-\beta_{1})/(1-\beta_{1}^{t})}{\epsilon+\sqrt{(\sum_{s=1}^{t}\beta_{2}^{t-s}\bg_{s}^{2}[i])(1-\beta_{2})/(1-\beta_{2}^{t})}},\forall i\in\left[d\right].\label{eq:equivalence-delta}
\end{equation}
Consequently, the iterates of $\EOTNC$ satisfy that, for any $t\in\left[T\right]$,
\[
\bx_{t+1}[i]=\bx_{t}[i]+\alpha_{t}\bdelta_{t+1}[i]=\bx_{t}[i]-\frac{\gamma\alpha_{t}(\sum_{s=1}^{t}\beta_{1}^{t-s}\bg_{s}[i])(1-\beta_{1})/(1-\beta_{1}^{t})}{\epsilon+\sqrt{(\sum_{s=1}^{t}\beta_{2}^{t-s}\bg_{s}^{2}[i])(1-\beta_{2})/(1-\beta_{2}^{t})}},\forall i\in\left[d\right].
\]
In particular, starting from the same initial point $\bx_{1}\in\R^{d}$,
the above recursion coincides with the trajectory of $\Adam$ with
the learning rate $\left\{ \gamma_{t}=\gamma\alpha_{t}\right\} _{t=1}^{T}$.
\end{proof}

\section{Theoretical Analysis\label{sec:analysis}}

In this section, we provide the theoretical analysis and prove Theorem
\ref{thm:Adam-rate} at the end. Compared with the existing literature,
our proofs require a more careful and technical analysis, since we
neither restrict the domain of $\FTRL$ to a bounded set, which is
equivalent to clipping $\bdelta_{t}$, as in \citet{pmlr-v235-ahn24b}
and \citet{NEURIPS2024_ac8ec9b4}, nor add additional regularization
to the loss function $\ell_{t}$, which is equivalent to modifying
the stepsize of $\Adam$, as in \citet{xie2026dynamic}, and, moreover,
we include the important bias-correction term that was, however, missed
in all prior works studying $\Adam$ from an online learning perspective
\citep{pmlr-v235-ahn24b,NEURIPS2024_ac8ec9b4,pmlr-v313-nguyen26b,xie2026dynamic}.
In addition, we highlight that our results are established under heavy-tailed
noise, which not only generalizes the widely considered finite-variance
condition but also more closely reflects practical settings.

\subsection{A Useful Auxiliary Sequence}

We begin by introducing an auxiliary sequence $\left\{ \eta_{t}\right\} _{t=1}^{T}$
defined as follows:
\begin{equation}
\eta_{t}\defeq\frac{\gamma(1-\beta_{1})/(1-\beta_{1}^{t})}{\sqrt{(1-\beta_{2})/(1-\beta_{2}^{t})}},\forall t\in\left[T\right].\label{eq:main-eta-def}
\end{equation}
As will become clear later, $\left\{ \eta_{t}\right\} _{t=1}^{T}$
plays an important role in our analysis. In particular, Lemma \ref{lem:eta-prop}
below establishes two key properties of this sequence, which will
be used in the subsequent proof.
\begin{lem}
\label{lem:eta-prop}Let $\left\{ \eta_{t}\right\} _{t=1}^{T}$ be
defined in (\ref{eq:main-eta-def}), then the following two properties
hold:
\begin{itemize}
\item $\eta_{t}\geq\rho\eta_{t+1},\forall t\in\left[T-1\right]$.
\item $\left\{ \eta_{t}\sqrt{1-\rho^{2t}}\right\} _{t=1}^{T}$ is a nondecreasing
sequence.
\end{itemize}
\end{lem}
\begin{proof}
See Appendix \ref{sec:eta}.
\end{proof}

\subsection{Two Intermediate Lemmas}

By Fact \ref{fact:equivalence}, proving convergence of $\Adam$ to
a $(\lambda,\varepsilon)$-$L_{1}$-stationary point reduces to analyzing
$\EOTNC$ with $\FTRL$ using the stepsize sequence $\left\{ \bet_{t}\right\} _{t=1}^{T}$
given in (\ref{eq:main-vector-eta-def}). Specifically, it suffices
to invoke Lemmas \ref{lem:EO2NC-gradient} and \ref{lem:EO2NC-stability}
for $\EOTNC$ and combine them with the regret bound for $\FTRL$
in Lemma \ref{lem:FTRL-regret}.

First, we recall that Lemma \ref{lem:FTRL-regret} relies on a technical
condition, i.e., the stepsize sequence $\left\{ \bet_{t}[i]\right\} _{t=1}^{T}$
is nonincreasing for every coordinate $i\in\left[d\right]$. In the
following Lemma \ref{lem:vector-eta-prop}, we verify that the stepsize
$\left\{ \bet_{t}\right\} _{t=1}^{T}$ specified in (\ref{eq:main-vector-eta-def})
indeed satisfies this requirement.
\begin{lem}
\label{lem:vector-eta-prop}Let $\left\{ \bet_{t}\right\} _{t=1}^{T}$
be defined in (\ref{eq:main-vector-eta-def}), then $\left\{ \bet_{t}[i]\right\} _{t=1}^{T}$
is nonincreasing for each $i\in\left[d\right]$.
\end{lem}
\begin{proof}
We fix a coordinate $i\in\left[d\right]$ in the following proof.
For any $t\in\left[T-1\right]$, note that
\[
\bet_{t}[i]\overset{(\ref{eq:main-vector-eta-def})}{=}\frac{\gamma\beta_{1}^{t}(1-\beta_{1})/(1-\beta_{1}^{t})}{\epsilon+\sqrt{(\sum_{s=1}^{t}\beta_{2}^{t-s}\bg_{s}^{2}[i])(1-\beta_{2})/(1-\beta_{2}^{t})}}\overset{\rho=\beta_{1}/\sqrt{\beta_{2}}}{=}\frac{\rho^{t}\eta_{t}}{\frac{\epsilon}{\sqrt{\beta_{2}^{t}(1-\beta_{2})/(1-\beta_{2}^{t})}}+\sqrt{\sum_{s=1}^{t}\beta_{2}^{-s}\bg_{s}^{2}[i]}}.
\]
Since both $\frac{\epsilon}{\sqrt{\beta_{2}^{t}(1-\beta_{2})/(1-\beta_{2}^{t})}}$
and $\sqrt{\sum_{s=1}^{t}\beta_{2}^{-s}\bg_{s}^{2}[i]}$ are nondecreasing
in $t$, it suffices to show $\rho^{t}\eta_{t}\geq\rho^{t+1}\eta_{t+1}$,
or equivalently, $\eta_{t}\geq\rho\eta_{t+1}$. The latter follows
from the first result in Lemma \ref{lem:eta-prop}.
\end{proof}

Next, we observe that Lemma \ref{lem:EO2NC-stability} depends on
the value of $\E_{\random}\left[\left\Vert \bdelta_{t+1}\right\Vert ^{2}\right]$,
where $\left\Vert \cdot\right\Vert $ can be any norm, and Lemma \ref{lem:FTRL-regret}
needs us to bound $\left\Vert \bdelta_{t}-\bdelta_{t+1}\right\Vert _{\infty}$.
To handle these terms, we prove a stronger result that $\left\Vert \bdelta_{t+1}\right\Vert _{\infty}$
is bounded above by a nondecreasing sequence when $\rho\in\left[0,1\right)$,
and this sequence is itself uniformly upper bounded. Surprisingly,
to the best of our knowledge, this important feature has not been
discussed in the literature.
\begin{lem}
\label{lem:iterate-bound}With the loss $\left\{ \ell_{t}(\bdelta)=\left\langle \beta_{1}^{-t}\bg_{t},\bdelta\right\rangle \right\} _{t=1}^{T}$
and the stepsize $\left\{ \bet_{t}\right\} _{t=1}^{T}$ defined in
(\ref{eq:main-vector-eta-def}), $\FTRL$ (Algorithm \ref{alg:FTRL})
guarantees that $\left\Vert \bdelta_{t+1}\right\Vert _{\infty}\leq D_{t+1}\leq D_{\infty},\forall t\in\left[T\right]$,
where $\left\{ D_{t+1}\defeq\eta_{t}\sqrt{\frac{1-\rho^{2t}}{1-\rho^{2}}}\right\} _{t=1}^{T}$
is a nondecreasing sequence and $D_{\infty}\defeq\lim_{t\to\infty}D_{t+1}=\frac{\gamma(1-\beta_{1})}{\sqrt{(1-\beta_{2})(1-\rho^{2})}}.$
\end{lem}
\begin{proof}
Given $t\in\left[T\right]$, recall that
\[
\bdelta_{t+1}[i]\overset{(\ref{eq:equivalence-delta})}{=}-\frac{\gamma(\sum_{s=1}^{t}\beta_{1}^{t-s}\bg_{s}[i])(1-\beta_{1})/(1-\beta_{1}^{t})}{\epsilon+\sqrt{(\sum_{s=1}^{t}\beta_{2}^{t-s}\bg_{s}^{2}[i])(1-\beta_{2})/(1-\beta_{2}^{t})}},\forall i\in\left[d\right],
\]
which implies that, for any $i\in\left[d\right]$,
\begin{align*}
\left|\bdelta_{t+1}[i]\right| & \leq\frac{\gamma\left|\sum_{s=1}^{t}\beta_{1}^{t-s}\bg_{s}[i]\right|(1-\beta_{1})/(1-\beta_{1}^{t})}{\epsilon+\sqrt{(\sum_{s=1}^{t}\beta_{2}^{t-s}\bg_{s}^{2}[i])(1-\beta_{2})/(1-\beta_{2}^{t})}}\overset{\epsilon\geq0}{\leq}\frac{\gamma\left|\sum_{s=1}^{t}\beta_{1}^{t-s}\bg_{s}[i]\right|(1-\beta_{1})/(1-\beta_{1}^{t})}{\sqrt{(\sum_{s=1}^{t}\beta_{2}^{t-s}\bg_{s}^{2}[i])(1-\beta_{2})/(1-\beta_{2}^{t})}}\\
 & \overset{(\ref{eq:main-eta-def})}{=}\eta_{t}\frac{\left|\sum_{s=1}^{t}\beta_{1}^{t-s}\bg_{s}[i]\right|}{\sqrt{\sum_{s=1}^{t}\beta_{2}^{t-s}\bg_{s}^{2}[i]}}\leq\eta_{t}\sqrt{\sum_{s=1}^{t}\rho^{2(t-s)}}=\eta_{t}\sqrt{\frac{1-\rho^{2t}}{1-\rho^{2}}},
\end{align*}
where the second-to-last step is due to Cauchy-Schwarz inequality.
Therefore,
\[
\left\Vert \bdelta_{t+1}\right\Vert _{\infty}\leq D_{t+1},\forall t\in\left[T\right].
\]
The monotonicity of $\left\{ D_{t+1}\right\} _{t=1}^{T}$ holds due
to the second result in Lemma \ref{lem:eta-prop}. Finally, $D_{\infty}$
is obtained by taking the limit of $D_{t+1}$. 
\end{proof}

\subsection{Two Core Results}

Equipped with Lemmas \ref{lem:vector-eta-prop} and \ref{lem:iterate-bound},
we are ready to prove two core results that will be used to complete
the convergence analysis.

First, we establish the following regret bound for $\FTRL$, whose
proof relies on combining Lemmas \ref{lem:FTRL-regret}, \ref{lem:vector-eta-prop},
and \ref{lem:iterate-bound}.
\begin{lem}
\label{lem:regret}Under Assumptions \ref{assu:lip} and \ref{assu:heavy},
with the loss $\left\{ \ell_{t}(\bdelta)=\left\langle \beta_{1}^{-t}\bg_{t},\bdelta\right\rangle \right\} _{t=1}^{T}$
and the stepsize $\left\{ \bet_{t}\right\} _{t=1}^{T}$ defined  in
(\ref{eq:main-vector-eta-def}), $\FTRL$ (Algorithm \ref{alg:FTRL})
guarantees that, for any $t\in\left[T\right]$,
\[
\E_{\random}\left[\beta_{1}^{t}\reg_{t}(\bu_{t})\right]\leq\frac{U^{2}d\epsilon(1-\beta_{1}^{t})}{2\gamma(1-\beta_{1})}+\left(\frac{U^{2}\sqrt{1-\beta_{2}}(1-\beta_{1}^{t})}{\sqrt{2}\gamma(1-\beta_{1})}+\frac{8\gamma(1-\beta_{1})}{\rho\sqrt{(1-\rho^{2})(1-\beta_{2})}}\right)\left(\frac{\left\Vert \bG\right\Vert _{1}}{\sqrt{1-\beta_{2}}}+\frac{2\left\Vert \bsig\right\Vert _{1}}{(1-\beta_{2})^{\frac{1}{\p}}}\right),
\]
where $\left\{ \bu_{t}\right\} _{t=1}^{T}$ is defined in (\ref{eq:main-u-def}).
\end{lem}
\begin{proof}
By Lemma \ref{lem:vector-eta-prop}, $\left\{ \bet_{t}[i]\right\} _{t=1}^{T}$
is a nonincreasing sequence for any $i\in\left[d\right]$, which allows
us to invoke Lemma \ref{lem:FTRL-regret} to obtain that, for any
$t\in\left[T\right]$,
\begin{align*}
\reg_{t}(\bu_{t}) & \leq\frac{1}{2}\left\Vert \bu_{t}\right\Vert _{\bet_{t}^{-1}}^{2}+\sum_{i=1}^{d}\sum_{s=1}^{t}\left\Vert \bdelta_{s}-\bdelta_{s+1}\right\Vert _{\infty}\left|\nabla_{i}\ell_{s}(\bdelta_{s})\right|\land\frac{\bet_{s-1}[i]\left(\nabla_{i}\ell_{s}(\bdelta_{s})\right)^{2}}{2}\\
 & \overset{(a)}{=}\frac{1}{2}\left\Vert \bu_{t}\right\Vert _{\bet_{t}^{-1}}^{2}+\sum_{i=1}^{d}\sum_{s=1}^{t}\left\Vert \bdelta_{s}-\bdelta_{s+1}\right\Vert _{\infty}\rho^{-s}\beta_{2}^{-\frac{s}{2}}\left|\bg_{s}[i]\right|\land\frac{\bet_{s-1}[i]\cdot\rho^{-2s}\beta_{2}^{-s}\bg_{s}^{2}[i]}{2}\\
 & \overset{(b)}{\leq}\frac{1}{2}\left\Vert \bu_{t}\right\Vert _{\bet_{t}^{-1}}^{2}+\rho^{-t}\sum_{i=1}^{d}\sum_{s=1}^{t}\left\Vert \bdelta_{s}-\bdelta_{s+1}\right\Vert _{\infty}\beta_{2}^{-\frac{s}{2}}\left|\bg_{s}[i]\right|\land\frac{\bet_{s-1}[i]\cdot\rho^{-s}\beta_{2}^{-s}\bg_{s}^{2}[i]}{2},
\end{align*}
where $(a)$ follows from $\nabla_{i}\ell_{s}(\bdelta_{s})=\beta_{1}^{-s}\bg_{s}[i]\overset{(\ref{eq:main-rho-def})}{=}\rho^{-s}\beta_{2}^{-\frac{s}{2}}\bg_{s}[i]$
and $(b)$ holds by $\rho\overset{(\ref{eq:main-rho-range})}{<}1$.
Multiplying both sides of the above inequality by $\beta_{1}^{t}$
and using $\beta_{1}\overset{(\ref{eq:main-rho-def})}{=}\rho\sqrt{\beta_{2}}$,
we get for any $t\in\left[T\right]$,
\begin{equation}
\beta_{1}^{t}\reg_{t}(\bu_{t})\leq\frac{\beta_{1}^{t}}{2}\left\Vert \bu_{t}\right\Vert _{\bet_{t}^{-1}}^{2}+\beta_{2}^{\frac{t}{2}}\sum_{i=1}^{d}\sum_{s=1}^{t}\left\Vert \bdelta_{s}-\bdelta_{s+1}\right\Vert _{\infty}\beta_{2}^{-\frac{s}{2}}\left|\bg_{s}[i]\right|\land\frac{\bet_{s-1}[i]\cdot\rho^{-s}\beta_{2}^{-s}\bg_{s}^{2}[i]}{2}.\label{eq:regret-1}
\end{equation}

Fix $t\in\left[T\right]$ in the following proof. We first observe
that
\begin{align}
\frac{\beta_{1}^{t}}{2}\left\Vert \bu_{t}\right\Vert _{\bet_{t}^{-1}}^{2} & =\frac{\beta_{1}^{t}}{2}\sum_{i=1}^{d}\frac{\bu_{t}^{2}[i]}{\bet_{t}[i]}\overset{(\ref{eq:main-u-def})}{\leq}\frac{\beta_{1}^{t}U^{2}}{2}\sum_{i=1}^{d}\frac{1}{\bet_{t}[i]}\nonumber \\
 & \overset{(\ref{eq:main-vector-eta-def})}{=}\frac{U^{2}}{2}\sum_{i=1}^{d}\frac{\epsilon+\sqrt{(\sum_{s=1}^{t}\beta_{2}^{t-s}\bg_{s}^{2}[i])(1-\beta_{2})/(1-\beta_{2}^{t})}}{\gamma(1-\beta_{1})/(1-\beta_{1}^{t})}\nonumber \\
 & \overset{(\ref{eq:main-eta-def})}{=}\frac{U^{2}d\epsilon(1-\beta_{1}^{t})}{2\gamma(1-\beta_{1})}+\frac{U^{2}}{2\eta_{t}}\sum_{i=1}^{d}\sqrt{\sum_{s=1}^{t}\beta_{2}^{t-s}\bg_{s}^{2}[i]}.\label{eq:regret-2}
\end{align}

Next, we recall from Lemma \ref{lem:iterate-bound} that $\left\Vert \bdelta_{t+1}\right\Vert \leq D_{t+1},\forall t\in\left[T\right]$
where $\left\{ D_{t+1}=\eta_{t}\sqrt{\frac{1-\rho^{2t}}{1-\rho^{2}}}\right\} _{t=1}^{T}$
is a nondecreasing sequence. Therefore, given $s\in\left[t\right]$,
we notice that
\[
\left\Vert \bdelta_{s}-\bdelta_{s+1}\right\Vert _{\infty}\leq\left\Vert \bdelta_{s}\right\Vert _{\infty}+\left\Vert \bdelta_{s+1}\right\Vert _{\infty}\leq\1\left[s>1\right]D_{s}+D_{s+1}\leq2D_{t+1},
\]
where we use $\bdelta_{1}=\bzero$ in the second step (see Remark
\ref{rem:FTRL-delta-eta}). Thus, we have
\begin{equation}
\left\Vert \bdelta_{s}-\bdelta_{s+1}\right\Vert _{\infty}\beta_{2}^{-\frac{s}{2}}\left|\bg_{s}[i]\right|\leq2D_{t+1}\beta_{2}^{-\frac{s}{2}}\left|\bg_{s}[i]\right|.\label{eq:regret-3}
\end{equation}
Moreover, by $\epsilon\geq0$ and the definitions of $\left\{ \bet_{t}\right\} _{t=1}^{T}$
(see (\ref{eq:main-vector-eta-def})) and $\left\{ \eta_{t}\right\} _{t=1}^{T}$
(see (\ref{eq:main-eta-def})), we have
\[
\bet_{t}[i]\leq\frac{\eta_{t}\rho^{t}}{\sqrt{\sum_{s=1}^{t}\beta_{2}^{-s}\bg_{s}^{2}[i]}},\forall t\in\left[T\right],
\]
which implies that, for $2\leq s\leq t$,
\begin{equation}
\frac{\bet_{s-1}[i]\cdot\rho^{-s}\beta_{2}^{-s}\bg_{s}^{2}[i]}{2}\leq\frac{\eta_{s-1}\beta_{2}^{-s}\bg_{s}^{2}[i]}{2\rho\sqrt{\sum_{l=1}^{s-1}\beta_{2}^{-l}\bg_{l}^{2}[i]}}\leq\frac{D_{t+1}\beta_{2}^{-s}\bg_{s}^{2}[i]}{2\rho\sqrt{\sum_{l=1}^{s-1}\beta_{2}^{-l}\bg_{l}^{2}[i]}},\label{eq:regret-4}
\end{equation}
where the last step is due to $\eta_{s-1}\leq D_{s}\leq D_{t+1}$.

We combine (\ref{eq:regret-3}) and (\ref{eq:regret-4}) to obtain
that, for $2\leq s\leq t$,
\begin{align*}
 & \left\Vert \bdelta_{s}-\bdelta_{s+1}\right\Vert _{\infty}\beta_{2}^{-\frac{s}{2}}\left|\bg_{s}[i]\right|\land\frac{\bet_{s-1}[i]\cdot\rho^{-s}\beta_{2}^{-s}\bg_{s}^{2}[i]}{2}\\
\leq & 2D_{t+1}\left(\beta_{2}^{-\frac{s}{2}}\left|\bg_{s}[i]\right|\land\frac{\beta_{2}^{-s}\bg_{s}^{2}[i]}{2\rho\sqrt{\sum_{l=1}^{s-1}\beta_{2}^{-l}\bg_{l}^{2}[i]}}\right)\leq\frac{2\sqrt{2}D_{t+1}\beta_{2}^{-s}\bg_{s}^{2}[i]}{\rho\sqrt{\sum_{l=1}^{s}\beta_{2}^{-l}\bg_{l}^{2}[i]}}.
\end{align*}
By (\ref{eq:regret-3}), the above bound also holds for $s=1$. Hence,
we can find that
\begin{align}
 & \sum_{s=1}^{t}\left\Vert \bdelta_{s}-\bdelta_{s+1}\right\Vert _{\infty}\beta_{2}^{-\frac{s}{2}}\left|\bg_{s}[i]\right|\land\frac{\bet_{s-1}[i]\cdot\rho^{-s}\beta_{2}^{-s}\bg_{s}^{2}[i]}{2}\nonumber \\
\leq & \sum_{s=1}^{t}\frac{2\sqrt{2}D_{t+1}\beta_{2}^{-s}\bg_{s}^{2}[i]}{\rho\sqrt{\sum_{l=1}^{s}\beta_{2}^{-l}\bg_{l}^{2}[i]}}\leq\frac{4\sqrt{2}D_{t+1}}{\rho}\sqrt{\sum_{s=1}^{t}\beta_{2}^{-s}\bg_{s}^{2}[i]}.\label{eq:regret-5}
\end{align}

We plug (\ref{eq:regret-2}) and (\ref{eq:regret-5}) back into (\ref{eq:regret-1})
to have
\begin{align*}
\beta_{1}^{t}\reg_{t}(\bu_{t}) & \leq\frac{U^{2}d\epsilon(1-\beta_{1}^{t})}{2\gamma(1-\beta_{1})}+\frac{U^{2}}{2\eta_{t}}\sum_{i=1}^{d}\sqrt{\sum_{s=1}^{t}\beta_{2}^{t-s}\bg_{s}^{2}[i]}+\beta_{2}^{\frac{t}{2}}\sum_{i=1}^{d}\frac{4\sqrt{2}D_{t+1}}{\rho}\sqrt{\sum_{s=1}^{t}\beta_{2}^{-s}\bg_{s}^{2}[i]}\\
 & =\frac{U^{2}d\epsilon(1-\beta_{1}^{t})}{2\gamma(1-\beta_{1})}+\left(\frac{U^{2}}{2\eta_{t}}+\frac{4\sqrt{2}D_{t+1}}{\rho}\right)\sum_{i=1}^{d}\sqrt{\sum_{s=1}^{t}\beta_{2}^{t-s}\bg_{s}^{2}[i]},
\end{align*}
which implies that
\begin{equation}
\E\left[\beta_{1}^{t}\reg_{t}(\bu_{t})\right]\leq\frac{U^{2}d\epsilon(1-\beta_{1}^{t})}{2\gamma(1-\beta_{1})}+\left(\frac{U^{2}}{2\eta_{t}}+\frac{4\sqrt{2}D_{t+1}}{\rho}\right)\sum_{i=1}^{d}\E\left[\sqrt{\sum_{s=1}^{t}\beta_{2}^{t-s}\bg_{s}^{2}[i]}\right].\label{eq:regret-6}
\end{equation}
Now, let $\bxi_{t}\defeq\bg_{t}-\nabla F(\bx_{t}),\forall t\in\left[T\right]$
denote the gradient noise. We note that, for any $i\in\left[d\right]$,
\begin{align}
 & \E\left[\sqrt{\sum_{s=1}^{t}\beta_{2}^{t-s}\bg_{s}^{2}[i]}\right]=\E\left[\sqrt{\sum_{s=1}^{t}\beta_{2}^{t-s}\left(\nabla_{i}F(\bx_{s})+\bxi_{s}[i]\right)^{2}}\right]\overset{(a)}{\leq}\E\left[\sqrt{2\sum_{s=1}^{t}\beta_{2}^{t-s}\left(\bG^{2}[i]+\bxi_{s}^{2}[i]\right)}\right]\nonumber \\
\overset{(b)}{\leq} & \sqrt{2\sum_{s=1}^{t}\beta_{2}^{t-s}\bG^{2}[i]}+\E\left[\sqrt{2\sum_{s=1}^{t}\beta_{2}^{t-s}\bxi_{s}^{2}[i]}\right]\overset{(c)}{\leq}\sqrt{2\frac{1-\beta_{2}^{t}}{1-\beta_{2}}}\bG[i]+\sqrt{2}\E\left[\left(\sum_{s=1}^{t}\beta_{2}^{\frac{\p(t-s)}{2}}\left|\bxi_{s}[i]\right|^{\p}\right)^{\frac{1}{\p}}\right]\nonumber \\
\overset{(d)}{\leq} & \sqrt{2\frac{1-\beta_{2}^{t}}{1-\beta_{2}}}\bG[i]+\sqrt{2}\left(\sum_{s=1}^{t}\beta_{2}^{\frac{\p(t-s)}{2}}\E\left[\left|\bxi_{s}[i]\right|^{\p}\right]\right)^{\frac{1}{\p}}\overset{(e)}{\leq}\sqrt{2\frac{1-\beta_{2}^{t}}{1-\beta_{2}}}\bG[i]+2\sqrt{2}\frac{\sqrt{1-\beta_{2}^{t}}}{(1-\beta_{2})^{\frac{1}{\p}}}\bsig[i],\label{eq:regret-7}
\end{align}
where $(a)$ is by $\left(\nabla_{i}F(\bx_{s})+\bxi_{s}[i]\right)^{2}\leq2\left(\nabla_{i}F(\bx_{s})\right)^{2}+2\bxi_{s}^{2}[i]\overset{\text{Assumption }\ref{assu:lip}}{\leq}2\bG^{2}[i]+2\bxi_{s}^{2}[i]$,
$(b)$ holds due to $\sqrt{x+y}\leq\sqrt{x}+\sqrt{y},\forall x,y\geq0$,
$(c)$ follows from $\sum_{s=1}^{t}\beta_{2}^{t-s}=\frac{1-\beta_{2}^{t}}{1-\beta_{2}}$
and $\left\Vert \cdot\right\Vert _{2}\leq\left\Vert \cdot\right\Vert _{\p},\forall\p\in\left[1,2\right]$,
$(d)$ is due to H\"{o}lder's inequality, and $(e)$ is by Assumption
\ref{assu:heavy}, $\sum_{s=1}^{t}\beta_{2}^{\frac{\p(t-s)}{2}}=\frac{1-\beta_{2}^{\frac{\p t}{2}}}{1-\beta_{2}^{\frac{\p}{2}}}$,
$\left(1-\beta_{2}^{\frac{\p t}{2}}\right)^{\frac{1}{\p}}\leq\left(1-\beta_{2}^{t}\right)^{\frac{1}{\p}}\leq\sqrt{1-\beta_{2}^{t}}$,
$1-\beta_{2}^{\frac{\p}{2}}\geq\frac{\p}{2}\left(1-\beta_{2}\right)$,
and $\left(\frac{\p}{2}\right)^{\frac{1}{\p}}\geq\frac{1}{2}$ when
$\p\in\left[1,2\right]$.

Finally, combining (\ref{eq:regret-6}) and (\ref{eq:regret-7}),
we have
\[
\E\left[\beta_{1}^{t}\reg_{t}(\bu_{t})\right]\leq\frac{U^{2}d\epsilon(1-\beta_{1}^{t})}{2\gamma(1-\beta_{1})}+\left(\frac{U^{2}\sqrt{1-\beta_{2}^{t}}}{\sqrt{2}\eta_{t}}+\frac{8D_{t+1}\sqrt{1-\beta_{2}^{t}}}{\rho}\right)\left(\frac{\left\Vert \bG\right\Vert _{1}}{\sqrt{1-\beta_{2}}}+\frac{2\left\Vert \bsig\right\Vert _{1}}{(1-\beta_{2})^{\frac{1}{\p}}}\right).
\]
To finish the proof, we apply the following two results:
\begin{eqnarray*}
\eta_{t}\overset{(\ref{eq:main-eta-def})}{=}\frac{\gamma(1-\beta_{1})/(1-\beta_{1}^{t})}{\sqrt{(1-\beta_{2})/(1-\beta_{2}^{t})}} & \text{and} & D_{t+1}\sqrt{1-\beta_{2}^{t}}\leq D_{t+1}\overset{\text{Lemma }\ref{lem:iterate-bound}}{\leq}D_{\infty}=\frac{\gamma(1-\beta_{1})}{\sqrt{(1-\beta_{2})(1-\rho^{2})}}.
\end{eqnarray*}
\end{proof}

Next, to connect our analysis with the definition of $(\lambda,\varepsilon)$-$L_{1}$-stationary
point, we need the following Lemma \ref{lem:norm}, which is a direct
consequence of Lemma \ref{lem:iterate-bound}.
\begin{lem}
\label{lem:norm}With the loss $\left\{ \ell_{t}(\bdelta)=\left\langle \beta_{1}^{-t}\bg_{t},\bdelta\right\rangle \right\} _{t=1}^{T}$
and the stepsize $\left\{ \bet_{t}\right\} _{t=1}^{T}$ defined in
(\ref{eq:main-vector-eta-def}), $\FTRL$ (Algorithm \ref{alg:FTRL})
guarantees that, for any $t\in\left[T\right]$,
\[
\E_{\random}\left[\left\Vert \bdelta_{t+1}\right\Vert _{2}^{2}\right]\leq\frac{d\gamma^{2}(1-\beta_{1})^{2}}{(1-\beta_{2})(1-\rho^{2})}.
\]
\end{lem}
\begin{proof}
For any $t\in\left[T\right]$, note that $\left\Vert \bdelta_{t+1}\right\Vert _{2}^{2}\leq d\left\Vert \bdelta_{t+1}\right\Vert _{\infty}^{2}\leq dD_{\infty}^{2}=\frac{d\gamma^{2}(1-\beta_{1})^{2}}{(1-\beta_{2})(1-\rho^{2})}$,
where the second inequality follows from Lemma \ref{lem:iterate-bound}.
Taking expectations with respect to $\random$ on both sides then
gives the desired result.
\end{proof}

\subsection{Proof of Theorem \ref{thm:Adam-rate}}

With the preceding preparation, we are now ready to prove Theorem
\ref{thm:Adam-rate}.

\begin{proof}
We first recall two pieces of notation introduced earlier.
\begin{itemize}
\item $\left\{ \bar{\bx}_{t}\right\} _{t=1}^{T}$ is the $\beta_{1}$-EMA
of $\left\{ \bx_{t}\right\} _{t=1}^{T}$, i.e.,
\begin{equation}
\bar{\bx}_{t}=\frac{1-\beta_{1}}{1-\beta_{1}^{t}}\sum_{s=1}^{t}\beta_{1}^{t-s}\bx_{s},\forall t\in\left[T\right].\label{eq:Adam-rate-x}
\end{equation}
\item $\left\{ \by_{t}\right\} _{t=1}^{T}$ is the sequence satisfying that
$\Pr\left[\by_{t}=\bx_{s}\right]=\frac{1-\beta_{1}}{1-\beta_{1}^{t}}\beta_{1}^{t-s},\forall s\in\left[t\right],\forall t\in\left[T\right]$.
Consequently,
\begin{equation}
\E_{\by_{t}}\left[\by_{t}\right]=\frac{1-\beta_{1}}{1-\beta_{1}^{t}}\sum_{s=1}^{t}\beta_{1}^{t-s}\bx_{s}\overset{(\ref{eq:Adam-rate-x})}{=}\bar{\bx}_{t},\forall t\in\left[T\right].\label{eq:Adam-rate-y}
\end{equation}
\end{itemize}
Now, by Definition \ref{def:stationary} and (\ref{eq:Adam-rate-y}),
we have
\begin{align}
 & \E_{\tau,\random}\left[\left\Vert \nabla F(\bar{\bx}_{\tau})\right\Vert _{1}^{\left[\lambda\right]}\right]\leq\E_{\tau,\random}\left[\left\Vert \E_{\by_{\tau}}\left[\nabla F(\by_{\tau})\right]\right\Vert _{1}+\lambda\E_{\by_{\tau}}\left[\left\Vert \by_{\tau}-\bar{\bx}_{\tau}\right\Vert _{2}^{2}\right]\right]\nonumber \\
\leq & \frac{\Delta}{UT}+\sum_{t=1}^{T}\frac{1-\beta_{1}\1\left[t\neq T\right]}{UT}\E_{\random}\left[\beta_{1}^{t}\reg_{t}(\bu_{t})\right]\nonumber \\
 & +\frac{2\left\Vert \bsig\right\Vert _{1}}{\left(1-\beta_{1}\right)^{\frac{1}{\p}}T}+2\left(1-\beta_{1}\right)^{\frac{1}{\cp}}\left\Vert \bsig\right\Vert _{1}+\frac{4\lambda\beta_{1}}{(1-\beta_{1})^{2}T}\sum_{t=1}^{T}\E_{\random}\left[\left\Vert \bdelta_{t+1}\right\Vert _{2}^{2}\right],\label{eq:Adam-rate-1}
\end{align}
where $\left\{ \bu_{t}\right\} _{t=1}^{T}$ is defined in (\ref{eq:main-u-def}),
and the second inequality follows from Lemmas \ref{lem:EO2NC-gradient}
and \ref{lem:EO2NC-stability}.

By Lemma \ref{lem:regret}, for any $t\in\left[T\right]$,
\[
\E_{\random}\left[\beta_{1}^{t}\reg_{t}(\bu_{t})\right]\leq\frac{U^{2}d\epsilon(1-\beta_{1}^{t})}{2\gamma(1-\beta_{1})}+\left(\frac{U^{2}\sqrt{1-\beta_{2}}(1-\beta_{1}^{t})}{\sqrt{2}\gamma(1-\beta_{1})}+\frac{8\gamma(1-\beta_{1})}{\rho\sqrt{(1-\rho^{2})(1-\beta_{2})}}\right)\left(\frac{\left\Vert \bG\right\Vert _{1}}{\sqrt{1-\beta_{2}}}+\frac{2\left\Vert \bsig\right\Vert _{1}}{(1-\beta_{2})^{\frac{1}{\p}}}\right),
\]
which implies that
\begin{align}
 & \sum_{t=1}^{T}\frac{1-\beta_{1}\1\left[t\neq T\right]}{UT}\E_{\random}\left[\beta_{1}^{t}\reg_{t}(\bu_{t})\right]\nonumber \\
\leq & \frac{Ud\epsilon}{2\gamma}+\left(\frac{U\sqrt{1-\beta_{2}}}{\sqrt{2}\gamma}+\frac{8\gamma(1-\beta_{1})(1-\beta_{1}+\frac{\beta_{1}}{T})}{U\rho\sqrt{(1-\rho^{2})(1-\beta_{2})}}\right)\left(\frac{\left\Vert \bG\right\Vert _{1}}{\sqrt{1-\beta_{2}}}+\frac{2\left\Vert \bsig\right\Vert _{1}}{(1-\beta_{2})^{\frac{1}{\p}}}\right).\label{eq:Adam-rate-2}
\end{align}

By Lemma \ref{lem:norm}, for any $t\in\left[T\right]$,
\[
\E_{\random}\left[\left\Vert \bdelta_{t+1}\right\Vert _{2}^{2}\right]\leq\frac{d\gamma^{2}(1-\beta_{1})^{2}}{(1-\beta_{2})(1-\rho^{2})},
\]
which implies that
\begin{equation}
\frac{4\lambda\beta_{1}}{(1-\beta_{1})^{2}T}\sum_{t=1}^{T}\E_{\random}\left[\left\Vert \bdelta_{t+1}\right\Vert _{2}^{2}\right]\leq\frac{4\lambda d\beta_{1}\gamma^{2}}{(1-\beta_{2})(1-\rho^{2})}.\label{eq:Adam-rate-3}
\end{equation}

Finally, plugging (\ref{eq:Adam-rate-2}) and (\ref{eq:Adam-rate-3})
back into (\ref{eq:Adam-rate-1}) gives the following bound
\begin{align*}
\E_{\tau,\random}\left[\left\Vert \nabla F(\bar{\bx}_{\tau})\right\Vert _{1}^{\left[\lambda\right]}\right]\leq & \frac{\Delta}{UT}+\frac{Ud\epsilon}{2\gamma}+\frac{2\left\Vert \bsig\right\Vert _{1}}{(1-\beta_{1})^{\frac{1}{\p}}T}+2(1-\beta_{1})^{\frac{1}{\cp}}\left\Vert \bsig\right\Vert _{1}+\frac{4\lambda d\beta_{1}\gamma^{2}}{(1-\beta_{2})(1-\rho^{2})}\\
 & +\left(\frac{U\sqrt{1-\beta_{2}}}{\sqrt{2}\gamma}+\frac{8\gamma(1-\beta_{1})(1-\beta_{1}+\frac{\beta_{1}}{T})}{U\rho\sqrt{(1-\rho^{2})(1-\beta_{2})}}\right)\left(\frac{\left\Vert \bG\right\Vert _{1}}{\sqrt{1-\beta_{2}}}+\frac{2\left\Vert \bsig\right\Vert _{1}}{(1-\beta_{2})^{\frac{1}{\p}}}\right),
\end{align*}
optimizing the R.H.S. of which with respect to $U$ yields the rate
\begin{align*}
\E_{\tau,\random}\left[\left\Vert \nabla F(\bar{\bx}_{\tau})\right\Vert _{1}^{\left[\lambda\right]}\right]\lesssim & \sqrt{\frac{\Delta d\epsilon}{\gamma T}}+\sqrt{\frac{(1-\beta_{1})(1-\beta_{1}+\frac{\beta_{1}}{T})}{\rho\sqrt{1-\rho^{2}}}}\left(\frac{\left\Vert \bG\right\Vert _{1}}{\sqrt{1-\beta_{2}}}+\frac{\left\Vert \bsig\right\Vert _{1}}{(1-\beta_{2})^{\frac{1}{\p}}}\right)\\
 & +\sqrt{\left(\frac{\sqrt{1-\beta_{2}}\Delta}{\gamma T}+\frac{(1-\beta_{1})(1-\beta_{1}+\frac{\beta_{1}}{T})d\epsilon}{\rho\sqrt{(1-\rho^{2})(1-\beta_{2})}}\right)\left(\frac{\left\Vert \bG\right\Vert _{1}}{\sqrt{1-\beta_{2}}}+\frac{\left\Vert \bsig\right\Vert _{1}}{(1-\beta_{2})^{\frac{1}{\p}}}\right)}\\
 & +\frac{\left\Vert \bsig\right\Vert _{1}}{(1-\beta_{1})^{\frac{1}{\p}}T}+(1-\beta_{1})^{\frac{1}{\cp}}\left\Vert \bsig\right\Vert _{1}+\frac{\lambda d\beta_{1}\gamma^{2}}{(1-\beta_{2})(1-\rho^{2})}.
\end{align*}
The above inequality leads us to the final desired bound by observing
the following fact:
\[
\sqrt{\frac{(1-\beta_{1})(1-\beta_{1}+\frac{\beta_{1}}{T})}{\rho\sqrt{1-\rho^{2}}}}\geq\frac{1-\beta_{1}}{\sqrt{\rho\sqrt{1-\rho^{2}}}}\overset{(\ref{eq:main-rho-range})}{\geq}1-\beta_{1}\geq(1-\beta_{1})^{\frac{1}{\cp}}\left(\frac{1-\beta_{2}}{2}\right)^{\frac{1}{\p}},
\]
where the last step is due to $\frac{1}{\p}+\frac{1}{\cp}=1$ and
\[
1-\beta_{1}>\frac{1-\beta_{1}^{2}}{2}\overset{(\ref{eq:main-rho-range})}{>}\frac{1-\beta_{2}}{2}.
\]
\end{proof}

%% file: conclusion.tex
\section{Conclusion}

In this work, we establish the first finite-time convergence rates
and asymptotic convergence guarantees for the original $\Adam$ algorithm
in nonsmooth nonconvex optimization. Our results provably apply to
heavy-tailed noise with tail index $\p\in\left(\nicefrac{4}{3},2\right]$,
even if the value of $\p$ is unknown. Moreover, many of our results
are proved under the choice $\beta_{1}=\beta_{2}$, aligning with
the recent empirical findings of \citet{NEURIPS2025_5bd9aa20}. Our
analysis builds on the $\EOTNC$ framework developed by \citet{pmlr-v235-zhang24k}
and \citet{NEURIPS2024_ac8ec9b4}. However, our proofs require a more
delicate and technical analysis than what is available in the existing
literature.

Several important questions remain open, which we list below.

First, our analysis requires a randomly scaled learning rate for $\Adam$,
which is unfortunately not used in practice. It is unclear whether
one can derandomize the learning rate and obtain a deterministic schedule
while still ensuring convergence.

Next, our finite-time rates, for both known and unknown $\p$, are
clearly suboptimal, since they lead to worse sample complexity than
the best-known bound achieved by $\ClippedAdam$ \citep{NEURIPS2024_ac8ec9b4}.
Therefore, there is substantial room for improvement.

Moreover, our results do not cover the regime where the tail index
$\p\in\left(1,\nicefrac{4}{3}\right]$. Determining whether $\Adam$,
even with a randomly scaled learning rate, can converge in this regime
remains an important problem, which we hope will be addressed in the
future.

Lastly, we note that our result also relies on a known time horizon
$T$. Developing a result that does not require prior knowledge of
$T$ would make the result more practical.

%% file: appendix.tex
\section{Proof of Lemma \ref{lem:EO2NC-gradient}\label{sec:EO2NC}}

\begin{proof}
In the following proof, we write $\bx_{0}\defeq\bx_{1}$. We also
recall from Remark \ref{rem:FTRL-delta-eta} that $\bdelta_{1}=\bzero$.
Thus, we can view $\bx_{1}$ as being generated by
\begin{eqnarray*}
\bx_{1}=\bx_{0}+\alpha_{0}\bdelta_{1} & \text{where} & \alpha_{0}\sim\ex(1).
\end{eqnarray*}
Equivalently, the update rules of both $\EOTNC$ (except for its Lines
3 and 4) and $\FTRL$ can be extended to $t=0$. For simplicity, we
denote the gradient noise by $\bxi_{t}\defeq\bg_{t}-\nabla F(\bx_{t})$
for any $t\in\left[T\right]$.

We first observe that the following identity holds
\[
F(\bx_{T})-F(\bx_{0})=\sum_{t=1}^{T}\left(1-\beta_{1}\1\left[t\neq T\right]\right)\beta_{1}^{t}\sum_{s=1}^{t}\beta_{1}^{-s}\left(F(\bx_{s})-F(\bx_{s-1})\right).
\]
Take expectations with respect to $\random$ on both sides of the
above equation and use 
\[
F(\bx_{T})-F(\bx_{0})\overset{\text{Assumption }\ref{assu:lb}}{\geq}F_{\star}-F(\bx_{0})=F_{\star}-F(\bx_{1})=-\Delta
\]
to have
\begin{equation}
-\Delta\leq\sum_{t=1}^{T}\left(1-\beta_{1}\1\left[t\neq T\right]\right)\beta_{1}^{t}\E_{\random}\left[\sum_{s=1}^{t}\beta_{1}^{-s}\left(F(\bx_{s})-F(\bx_{s-1})\right)\right].\label{eq:EO2NC-gradient-1}
\end{equation}

By \citet[Lemma 3.1]{pmlr-v235-zhang24k}, under Assumptions \ref{assu:FTC}
and \ref{assu:unbias}, $\EOTNC$ guarantees that
\[
\E_{\random}\left[F(\bx_{t})-F(\bx_{t-1})\right]=\E_{\random}\left[\left\langle \bg_{t},\bdelta_{t}\right\rangle \right],\forall t\in\left[T\right],
\]
which implies that, for any $t\in\left[T\right]$ and $\bu_{t}\in\R^{d}$,
\begin{align}
 & \E_{\random}\left[\sum_{s=1}^{t}\beta_{1}^{-s}\left(F(\bx_{s})-F(\bx_{s-1})\right)\right]=\E_{\random}\left[\sum_{s=1}^{t}\left\langle \beta_{1}^{-s}\bg_{s},\bdelta_{s}\right\rangle \right]\nonumber \\
= & \E_{\random}\left[\sum_{s=1}^{t}\left\langle \beta_{1}^{-s}\bg_{s},\bdelta_{s}-\bu_{t}\right\rangle \right]+\E_{\random}\left[\left\langle \sum_{s=1}^{t}\beta_{1}^{-s}\nabla F(\bx_{s}),\bu_{t}\right\rangle \right]+\E_{\random}\left[\left\langle \sum_{s=1}^{t}\beta_{1}^{-s}\bxi_{s},\bu_{t}\right\rangle \right]\nonumber \\
= & \E_{\random}\left[\reg_{t}(\bu_{t})\right]+\E_{\random}\left[\left\langle \sum_{s=1}^{t}\beta_{1}^{-s}\nabla F(\bx_{s}),\bu_{t}\right\rangle \right]+\E_{\random}\left[\left\langle \sum_{s=1}^{t}\beta_{1}^{-s}\bxi_{s},\bu_{t}\right\rangle \right],\label{eq:EO2NC-gradient-2}
\end{align}
where the second step uses $\bg_{s}=\nabla F(\bx_{s})+\bxi_{s}$.

Next, we use $\bu_{t}[i]\overset{(\ref{eq:main-u-def})}{=}-U\cdot\sgn\left(\sum_{s=1}^{t}\beta_{1}^{-s}\nabla_{i}F(\bx_{s})\right),\forall i\in\left[d\right]$
to obtain
\[
\left\langle \sum_{s=1}^{t}\beta_{1}^{-s}\nabla F(\bx_{s}),\bu_{t}\right\rangle =-U\left\Vert \sum_{s=1}^{t}\beta_{1}^{-s}\nabla F(\bx_{s})\right\Vert _{1}
\]
and
\[
\left\langle \sum_{s=1}^{t}\beta_{1}^{-s}\bxi_{s},\bu_{t}\right\rangle \leq\left\Vert \sum_{s=1}^{t}\beta_{1}^{-s}\bxi_{s}\right\Vert _{1}\left\Vert \bu_{t}\right\Vert _{\infty}=U\left\Vert \sum_{s=1}^{t}\beta_{1}^{-s}\bxi_{s}\right\Vert _{1}.
\]
By the two results above and (\ref{eq:EO2NC-gradient-2}), we get,
for any $t\in\left[T\right]$,
\begin{equation}
\E_{\random}\left[\sum_{s=1}^{t}\beta_{1}^{-s}\left(F(\bx_{s})-F(\bx_{s-1})\right)\right]\leq\E_{\random}\left[\reg_{t}(\bu_{t})\right]-U\E_{\random}\left[\left\Vert \sum_{s=1}^{t}\beta_{1}^{-s}\nabla F(\bx_{s})\right\Vert _{1}\right]+U\E_{\random}\left[\left\Vert \sum_{s=1}^{t}\beta_{1}^{-s}\bxi_{s}\right\Vert _{1}\right].\label{eq:EO2NC-gradient-3}
\end{equation}

Furthermore, we observe that
\begin{align}
 & \left\Vert \sum_{s=1}^{t}\beta_{1}^{-s}\nabla F(\bx_{s})\right\Vert _{1}=\left(\sum_{s^{\prime}=1}^{t}\beta_{1}^{-s^{\prime}}\right)\left\Vert \sum_{s=1}^{t}\frac{\beta_{1}^{-s}}{\sum_{s^{\prime}=1}^{t}\beta_{1}^{-s^{\prime}}}\nabla F(\bx_{s})\right\Vert _{1}\nonumber \\
= & \frac{\beta_{1}^{-t}-1}{1-\beta_{1}}\left\Vert \sum_{s=1}^{t}\frac{1-\beta_{1}}{1-\beta_{1}^{t}}\beta_{1}^{t-s}\nabla F(\bx_{s})\right\Vert _{1}=\frac{\beta_{1}^{-t}-1}{1-\beta_{1}}\left\Vert \E_{\by_{t}}\left[\nabla F(\by_{t})\right]\right\Vert _{1},\label{eq:EO2NC-gradient-4}
\end{align}
where the last step is due to the definition of $\by_{t}$. We can
also bound
\begin{align}
 & \E_{\random}\left[\left\Vert \sum_{s=1}^{t}\beta_{1}^{-s}\bxi_{s}\right\Vert _{1}\right]=\sum_{i=1}^{d}\E_{\random}\left[\left|\sum_{s=1}^{t}\beta_{1}^{-s}\bxi_{s}[i]\right|\right]\overset{(a)}{\leq}\sum_{i=1}^{d}\left(\E_{\random}\left[\left|\sum_{s=1}^{t}\beta_{1}^{-s}\bxi_{s}[i]\right|^{\p}\right]\right)^{\frac{1}{\p}}\nonumber \\
\overset{(b)}{\leq} & \sum_{i=1}^{d}\left(2^{2-\p}\sum_{s=1}^{t}\beta_{1}^{-\p s}\E_{\random}\left[\left|\bxi_{s}[i]\right|^{\p}\right]\right)^{\frac{1}{\p}}\overset{(c)}{\leq}2^{\frac{2}{\p}-1}\left(\sum_{s=1}^{t}\beta_{1}^{-\p s}\right)^{\frac{1}{\p}}\left\Vert \bsig\right\Vert _{1}\overset{\p>1}{\leq}2\left(\frac{\beta_{1}^{-\p t}-1}{1-\beta_{1}^{\p}}\right)^{\frac{1}{\p}}\left\Vert \bsig\right\Vert _{1},\label{eq:EO2NC-gradient-5}
\end{align}
where $(a)$ is by H\"{o}lder's inequality, $(b)$ follows from Assumption
\ref{assu:unbias} and the von Bahr-Esseen inequality \citep{8628f855-d581-37e5-a68f-5014db61e4b7}
for martingales \citep[Theorem 1.1 and Proposition 1.8]{10.15352/afa/06-4-1},
and $(c)$ holds due to Assumption \ref{assu:heavy}.

Finally, we plug (\ref{eq:EO2NC-gradient-4}) and (\ref{eq:EO2NC-gradient-5})
back into (\ref{eq:EO2NC-gradient-3}) to have
\[
\E_{\random}\left[\sum_{s=1}^{t}\beta_{1}^{-s}\left(F(\bx_{s})-F(\bx_{s-1})\right)\right]\leq\E_{\random}\left[\reg_{t}(\bu_{t})\right]-U\frac{\beta_{1}^{-t}-1}{1-\beta_{1}}\E_{\random}\left[\left\Vert \E_{\by_{t}}\left[\nabla F(\by_{t})\right]\right\Vert _{1}\right]+2U\left(\frac{\beta_{1}^{-\p t}-1}{1-\beta_{1}^{\p}}\right)^{\frac{1}{\p}}\left\Vert \bsig\right\Vert _{1}.
\]
 which, combined with (\ref{eq:EO2NC-gradient-1}), further implies
that
\begin{align*}
-\Delta & \leq\sum_{t=1}^{T}\left(1-\beta_{1}\1\left[t\neq T\right]\right)\left(\E_{\random}\left[\beta_{1}^{t}\reg_{t}(\bu_{t})\right]-U\frac{1-\beta_{1}^{t}}{1-\beta_{1}}\E_{\random}\left[\left\Vert \E_{\by_{t}}\left[\nabla F(\by_{t})\right]\right\Vert _{1}\right]+2U\left(\frac{1-\beta_{1}^{\p t}}{1-\beta_{1}^{\p}}\right)^{\frac{1}{\p}}\left\Vert \bsig\right\Vert _{1}\right)\\
 & \leq\sum_{t=1}^{T}\left(1-\beta_{1}\1\left[t\neq T\right]\right)\left(\E_{\random}\left[\beta_{1}^{t}\reg_{t}(\bu_{t})\right]-U\frac{1-\beta_{1}^{t}}{1-\beta_{1}}\E_{\random}\left[\left\Vert \E_{\by_{t}}\left[\nabla F(\by_{t})\right]\right\Vert _{1}\right]+\frac{2U\left\Vert \bsig\right\Vert _{1}}{(1-\beta_{1})^{\frac{1}{\p}}}\right).
\end{align*}
Rearranging terms, using the definition of $\tau$ (see Definition
\ref{def:tau}), and dividing both sides by $UT$, we obtain that
\[
\E_{\tau,\random}\left[\left\Vert \E_{\by_{\tau}}\left[\nabla F(\by_{\tau})\right]\right\Vert _{1}\right]\leq\frac{\Delta}{UT}+\sum_{t=1}^{T}\frac{1-\beta_{1}\1\left[t\neq T\right]}{UT}\E_{\random}\left[\beta_{1}^{t}\reg_{t}(\bu_{t})\right]+\frac{2\left\Vert \bsig\right\Vert _{1}}{(1-\beta_{1})^{\frac{1}{\p}}T}+2(1-\beta_{1})^{\frac{1}{\cp}}\left\Vert \bsig\right\Vert _{1}.
\]
\end{proof}

\section{Proof of Lemma \ref{lem:FTRL-regret}\label{sec:FTRL}}

\begin{proof}
Let
\[
L_{t}(\bdelta)\defeq\frac{1}{2}\left\Vert \bdelta\right\Vert _{\bet_{t-1}^{-1}}^{2}+\sum_{s=1}^{t-1}\ell_{s}(\bdelta),\forall t\in\left[T+1\right],
\]
where we recall from Remark \ref{rem:FTRL-delta-eta} that $\bdelta_{1}=\bzero$
and $\bet_{0}=\bet_{1}$. Equivalently, $\FTRL$ can be written as
\[
\bdelta_{t}=\argmin_{\bdelta\in\R^{d}}L_{t}(\bdelta),\forall t\in\left[T+1\right].
\]

By \citet[Lemma 7.1]{orabona2019modern}, the following equation holds
for any $t\in\left[T\right]$ and $\bu\in\R^{d}$,
\begin{align}
\reg_{t}(\bu) & =\sum_{s=1}^{t}\ell_{s}(\bdelta_{s})-\ell_{s}(\bu)\nonumber \\
 & =\frac{1}{2}\left\Vert \bu\right\Vert _{\bet_{t}^{-1}}^{2}+L_{t+1}(\bdelta_{t+1})-L_{t+1}(\bu)+\sum_{s=1}^{t}L_{s}(\bdelta_{s})-L_{s+1}(\bdelta_{s+1})+\ell_{s}(\bdelta_{s})\nonumber \\
 & \leq\frac{1}{2}\left\Vert \bu\right\Vert _{\bet_{t}^{-1}}^{2}+\sum_{s=1}^{t}L_{s}(\bdelta_{s})-L_{s+1}(\bdelta_{s+1})+\ell_{s}(\bdelta_{s}),\label{eq:FTRL-regret-1}
\end{align}
where the last step is due to $L_{t+1}(\bdelta_{t+1})=\min_{\bdelta\in\R^{d}}L_{t+1}(\bdelta)\leq L_{t+1}(\bu)$.
Note that for any $s\in\left[t\right]$,
\begin{align}
L_{s}(\bdelta_{s})-L_{s+1}(\bdelta_{s+1})+\ell_{s}(\bdelta_{s}) & =L_{s}(\bdelta_{s})-L_{s}(\bdelta_{s+1})+\frac{1}{2}\left\Vert \bdelta_{s+1}\right\Vert _{\bet_{s-1}^{-1}-\bet_{s}^{-1}}^{2}+\ell_{s}(\bdelta_{s})-\ell_{s}(\bdelta_{s+1})\nonumber \\
 & \overset{(a)}{\leq}L_{s}(\bdelta_{s})-L_{s}(\bdelta_{s+1})+\ell_{s}(\bdelta_{s})-\ell_{s}(\bdelta_{s+1})\nonumber \\
 & \overset{(b)}{\leq}\left\langle \nabla(L_{s}+\ell_{s})(\bdelta_{s}),\bdelta_{s}-\bdelta_{s+1}\right\rangle -\frac{1}{2}\left\Vert \bdelta_{s}-\bdelta_{s+1}\right\Vert _{\bet_{s-1}^{-1}}^{2}\nonumber \\
 & \overset{(c)}{=}\left\langle \nabla\ell_{s}(\bdelta_{s}),\bdelta_{s}-\bdelta_{s+1}\right\rangle -\frac{1}{2}\left\Vert \bdelta_{s}-\bdelta_{s+1}\right\Vert _{\bet_{s-1}^{-1}}^{2},\label{eq:FTRL-regret-2}
\end{align}
where $(a)$ holds due to $\bet_{0}=\bet_{1}$ and the condition that
$\left\{ \bet_{t}[i]\right\} _{t=1}^{T}$ is nonincreasing, $(b)$
follows from the strong convexity of $L_{s}+\ell_{s}$ due to the
assumption that $\left\{ \ell_{t}\right\} _{t=1}^{T}$ is differentiable
and convex, and $(c)$ is by $\nabla L_{s}(\bdelta_{s})=\bzero$ since
$\bdelta_{s}=\argmin_{\bdelta\in\R^{d}}L_{s}(\bdelta)$ and $L_{s}$
is differentiable.

We plug (\ref{eq:FTRL-regret-2}) back into (\ref{eq:FTRL-regret-1})
to obtain
\begin{align*}
\reg_{t}(\bu) & \leq\frac{1}{2}\left\Vert \bu\right\Vert _{\bet_{t}^{-1}}^{2}+\sum_{s=1}^{t}\left\langle \nabla\ell_{s}(\bdelta_{s}),\bdelta_{s}-\bdelta_{s+1}\right\rangle -\frac{1}{2}\left\Vert \bdelta_{s}-\bdelta_{s+1}\right\Vert _{\bet_{s-1}^{-1}}^{2}\\
 & =\frac{1}{2}\left\Vert \bu\right\Vert _{\bet_{t}^{-1}}^{2}+\sum_{i=1}^{d}\sum_{s=1}^{t}\nabla_{i}\ell_{s}(\bdelta_{s})\left(\bdelta_{s}[i]-\bdelta_{s+1}[i]\right)-\frac{\left(\bdelta_{s}[i]-\bdelta_{s+1}[i]\right)^{2}}{2\bet_{s-1}[i]}\\
 & \leq\frac{1}{2}\left\Vert \bu\right\Vert _{\bet_{t}^{-1}}^{2}+\sum_{i=1}^{d}\sum_{s=1}^{t}\left\Vert \bdelta_{s}-\bdelta_{s+1}\right\Vert _{\infty}\left|\nabla_{i}\ell_{s}(\bdelta_{s})\right|\land\frac{\bet_{s-1}[i]\left(\nabla_{i}\ell_{s}(\bdelta_{s})\right)^{2}}{2},
\end{align*}
where the last inequality follows from the following two bounds:
\begin{itemize}
\item Bound I:
\begin{align*}
 & \nabla_{i}\ell_{s}(\bdelta_{s})\left(\bdelta_{s}[i]-\bdelta_{s+1}[i]\right)-\frac{\left(\bdelta_{s}[i]-\bdelta_{s+1}[i]\right)^{2}}{2\bet_{s-1}[i]}\\
\leq & \nabla_{i}\ell_{s}(\bdelta_{s})\left(\bdelta_{s}[i]-\bdelta_{s+1}[i]\right)\leq\left\Vert \bdelta_{s}-\bdelta_{s+1}\right\Vert _{\infty}\left|\nabla_{i}\ell_{s}(\bdelta_{s})\right|.
\end{align*}
\item Bound II: By AM-GM inequality,
\[
\nabla_{i}\ell_{s}(\bdelta_{s})\left(\bdelta_{s}[i]-\bdelta_{s+1}[i]\right)-\frac{\left(\bdelta_{s}[i]-\bdelta_{s+1}[i]\right)^{2}}{2\bet_{s-1}[i]}\leq\frac{\bet_{s-1}[i]\left(\nabla_{i}\ell_{s}(\bdelta_{s})\right)^{2}}{2}.
\]
\end{itemize}
\end{proof}

\section{Proof of Lemma \ref{lem:eta-prop}\label{sec:eta}}

\begin{proof}
We fix $t\in\left[T\right]$ in the following proof.
\begin{itemize}
\item For the first inequality, by the definition of $\eta_{t}\overset{(\ref{eq:main-eta-def})}{=}\frac{\gamma(1-\beta_{1})/(1-\beta_{1}^{t})}{\sqrt{(1-\beta_{2})/(1-\beta_{2}^{t})}}$,
it is equivalent to show
\[
\frac{\sqrt{1-\beta_{2}^{t}}}{1-\beta_{1}^{t}}\geq\rho\frac{\sqrt{1-\beta_{2}^{t+1}}}{1-\beta_{1}^{t+1}}.
\]
After substituting $\beta_{1}\overset{(\ref{eq:main-rho-def})}{=}\sqrt{\beta_{2}}\rho$
and rearranging terms, we need to prove
\begin{equation}
\frac{1-\left(\sqrt{\beta_{2}}\rho\right)^{t+1}}{\rho\left(1-\left(\sqrt{\beta_{2}}\rho\right)^{t}\right)}\geq\sqrt{\frac{1-\beta_{2}^{t+1}}{1-\beta_{2}^{t}}}.\label{eq:eta-prop-1}
\end{equation}
Observe that
\[
\frac{1-\left(\sqrt{\beta_{2}}\rho\right)^{t+1}}{\rho\left(1-\left(\sqrt{\beta_{2}}\rho\right)^{t}\right)}=\frac{1-\sqrt{\beta_{2}}\rho}{\rho\left(1-\left(\sqrt{\beta_{2}}\rho\right)^{t}\right)}+\sqrt{\beta_{2}}=\frac{1}{\rho\left(\sum_{s=0}^{t-1}\left(\sqrt{\beta_{2}}\rho\right)^{s}\right)}+\sqrt{\beta_{2}},
\]
which implies that
\begin{equation}
\inf_{0\leq\rho<1}\frac{1-\left(\sqrt{\beta_{2}}\rho\right)^{t+1}}{\rho\left(1-\left(\sqrt{\beta_{2}}\rho\right)^{t}\right)}=\frac{1}{\sum_{s=0}^{t-1}\left(\sqrt{\beta_{2}}\right)^{s}}+\sqrt{\beta_{2}}=\frac{1-\beta_{2}^{\frac{t+1}{2}}}{1-\beta_{2}^{\frac{t}{2}}}.\label{eq:eta-prop-2}
\end{equation}
By (\ref{eq:eta-prop-1}) and (\ref{eq:eta-prop-2}), it suffices
to show
\[
\frac{1-\beta_{2}^{\frac{t+1}{2}}}{1-\beta_{2}^{\frac{t}{2}}}\geq\sqrt{\frac{1-\beta_{2}^{t+1}}{1-\beta_{2}^{t}}}\Leftrightarrow\sqrt{\frac{1-\beta_{2}^{\frac{t+1}{2}}}{1-\beta_{2}^{\frac{t}{2}}}}\geq\sqrt{\frac{1+\beta_{2}^{\frac{t+1}{2}}}{1+\beta_{2}^{\frac{t}{2}}}}\Leftrightarrow\beta_{2}^{\frac{t}{2}}\geq\beta_{2}^{\frac{t+1}{2}}\Leftrightarrow1\geq\beta_{2}.
\]
Therefore, the desired inequality follows.
\item For the second inequality, we plug in $\eta_{t}\overset{(\ref{eq:main-eta-def})}{=}\frac{\gamma(1-\beta_{1})/(1-\beta_{1}^{t})}{\sqrt{(1-\beta_{2})/(1-\beta_{2}^{t})}}$
and $\beta_{1}\overset{(\ref{eq:main-rho-def})}{=}\sqrt{\beta_{2}}\rho$
to have
\[
\eta_{t}\sqrt{1-\rho^{2t}}=\frac{\gamma(1-\beta_{1})}{\sqrt{1-\beta_{2}}}\sqrt{\frac{\left(1-\beta_{2}^{t}\right)\left(1-\rho^{2t}\right)}{\left(1-\left(\sqrt{\beta_{2}}\rho\right)^{t}\right)}}=\frac{\gamma(1-\beta_{1})}{\sqrt{1-\beta_{2}}}\sqrt{1-\left(\frac{\sqrt{\beta_{2}}^{t}-\rho^{t}}{1-\left(\sqrt{\beta_{2}}\rho\right)^{t}}\right)^{2}}.
\]
For simplicity, we denote by $a\defeq\sqrt{\beta_{2}}\lor\rho\in\left[0,1\right)$
and $b\defeq\sqrt{\beta_{2}}\land\rho\in\left[0,1\right)$. Thus,
it remains to show that the sequence $\left\{ \frac{a^{t}-b^{t}}{1-a^{t}b^{t}}\right\} _{t\in\N}$
is nonincreasing, which holds by Lemma \ref{lem:sequence}.
\end{itemize}
\end{proof}

\begin{lem}
\label{lem:sequence}Given $0\leq b\leq a<1$, the sequence $\left\{ \frac{a^{t}-b^{t}}{1-a^{t}b^{t}}\right\} _{t\in\N}$
is nonincreasing.
\end{lem}
\begin{proof}
If $b=0$ or $b=a$, the result immediately holds. Therefore, we only
need to consider the case $0<b<a<1$. Let $g(t)\defeq\frac{a^{t}-b^{t}}{1-a^{t}b^{t}}$,
we can find 
\begin{align*}
g^{\prime}(t) & =\frac{(a^{2t}-1)b^{t}\ln b-(b^{2t}-1)a^{t}\ln a}{(1-a^{t}b^{t})^{2}}\\
 & =\frac{a^{t}b^{t}(\ln a)(\ln b)t\left(\frac{a^{2t}-1}{a^{t}\ln a^{t}}-\frac{b^{2t}-1}{b^{t}\ln b^{t}}\right)}{(1-a^{t}b^{t})^{2}}\\
 & =\frac{a^{t}b^{t}(\ln a)(\ln b)t(h(a^{t})-h(b^{t}))}{(1-a^{t}b^{t})^{2}},
\end{align*}
where $h(x)\defeq\frac{x^{2}-1}{x\ln x}$. So it suffices to show
$h(x)$ is nonincreasing on $(0,1)$. Equivalently, we need to prove
$h^{\prime}(x)=\frac{-x^{2}+x^{2}\ln x+\ln x+1}{x^{2}\ln^{2}x}\leq0,\forall x\in(0,1)\Leftrightarrow\ln x\leq\frac{x^{2}-1}{x^{2}+1},\forall x\in(0,1)$,
which can be further reduced to show $\ln x\leq\frac{2(x-1)}{x+1},\forall x\in(0,1).$
Let $k(x)\defeq\frac{2(x-1)}{x+1}-\ln x$. We have $k^{\prime}(x)=\frac{4}{(x+1)^{2}}-\frac{1}{x}=-\frac{(x-1)^{2}}{x(x+1)^{2}}\leq0$.
So $k(x)\geq k(1)=0,\forall x\in(0,1)$, which implies that $\ln x\leq\frac{2(x-1)}{x+1},\forall x\in(0,1)$.
Therefore, the desired result holds.
\end{proof}